\documentclass[reqno]{au}
 \presentedby{\dots}
 \received{\dots}{\dots}

\usepackage{amsmath,amsfonts,amsthm,amssymb,amsthm}
\usepackage{xspace} 
\usepackage{graphicx} 
\usepackage{enumerate} 
\usepackage{url}

\allowdisplaybreaks 

\theoremstyle{plain}
\newtheorem*{thm*}{Theorem}
\newtheorem{thm}{Theorem}[section]
\newtheorem{cor}[thm]{Corollary}
\newtheorem*{lem*}{Lemma}
\newtheorem{lem}[thm]{Lemma}

\theoremstyle{definition}
\newtheorem{defn}[thm]{Definition}
\newtheorem{rmk}[thm]{Remark}
\newtheorem{ex}[thm]{Example}

\renewcommand{\=}{\approx}
\newcommand{\id}[3]{\ensuremath{#1\text{ }[#2\=#3]}}
\newcommand{\A}{\ensuremath{\mathbf{A}}\xspace}
\newcommand{\B}{\ensuremath{\mathbf{B}}\xspace}
\newcommand{\Vn}{\ensuremath{\class{V}_n}}

\newcommand{\op}[1]{\operatorname{#1}}

\newcommand{\Con}{\op{Con}}

\newcommand{\Pol}{\op{Pol}}
\newcommand{\Inv}{\op{Inv}}

\DeclareMathOperator{\CSP}{CSP}
\DeclareMathOperator{\SD}{SD}

\font\xxroman=cmr17 scaled \magstep1
\font\xviiroman=cmr17
\def\udot{\mathbin{\ooalign{$\cup$\crcr
   \hfil\raise2pt\hbox{\xviiroman.}\hfil\crcr}}}
\def\bigudotx#1#2{\mathop{\smash{\ooalign{$#1\bigcup$\crcr
  \hfil\raise 2pt\hbox{#2}\hfil\crcr}}\vphantom{\bigcup}}}
\def\bigudot{\mathop{\mathchoice%
 {\bigudotx\displaystyle{$\scriptscriptstyle\bullet$}}
 {\bigudotx\textstyle{\xxroman.}}
 {\relax}{\relax}}}

\DeclareMathAlphabet{\clss}{OT1}{pzc}{m}{n}
\newcommand\class[1]{\ensuremath{\clss{#1}}}
\DeclareMathAlphabet{\clsop}{OT1}{cmss}{bx}{n}
\newcommand\opV{{\clsop V}}

\newcommand\opS{{\clsop S}}
\newcommand\opP{{\clsop P}}
\newcommand\opH{{\clsop H}}

\newcommand\opSP{{\clsop {SP}}}

\begin{document}
\title[Commutative, idempotent groupoids and the CSP]{Commutative, idempotent groupoids and the constraint satisfaction problem}
\author[C. Bergman]{Clifford Bergman} 
\email{cbergman@iastate.edu}
\address{Department of Mathematics, Iowa State University, Ames, IA 50011}
\author[D. Failing]{David Failing}
\email{failida@quincy.edu}
\address{Department of Mathematics, Quincy University, Quincy, IL 62301}


\begin{abstract}
A restatement of the Algebraic Dichotomy Conjecture, due to Maroti and McKenzie, postulates that if a finite algebra $\A$ possesses a weak near-unanimity term, then the corresponding constraint satisfaction problem is tractable. A binary operation is weak near-unanimity if and only if it is both commutative and idempotent. Thus if the dichotomy conjecture is true, any finite commutative, idempotent groupoid (CI groupoid) will be tractable. It is known that every semilattice (i.e., an associative CI groupoid) is tractable.  A groupoid identity is of Bol-Moufang type if the same three variables appear on either side, one of the variables is repeated, the remaining two variables appear once, and the variables appear in the same order on either side (for example, $x(x(yz))\approx(x(xy))z$). These identities can be thought of as generalizations of associativity. We show that there are exactly 8 varieties of CI groupoids defined by a single additional identity of Bol-Moufang type, derive some of their important structural properties, and use that structure theory to show that 7 of the varieties are tractable. We also characterize the finite members of the variety of CI groupoids satisfying the self-distributive law $x(yz)\approx(xy)(xz)$, and show that they are tractable.
\end{abstract}

\maketitle

\section{Introduction}\label{sec:Introduction}
The goal in a Constraint Satisfaction Problem (CSP) is to determine if there is a suitable assignment of values to variables subject to constraints on their allowed simultaneous values. The CSP provides a common framework in which many important combinatorial problems may be formulated---for example, graph colorability or propositional satisfiability. It is also of great importance in theoretical computer science, where it is applied to problems as varied as database theory and natural language processing.

In what follows, we will assume $\mathbf{P}\neq\mathbf{NP}$. Problems in $\mathbf{P}$ are said to be tractable. The general CSP is known to be $\mathbf{NP}$-complete \cite{MackworthConsistency}. One focus of current research is on instances of the CSP in which the constraint relations are members of some fixed finite set of relations over a finite set. The goal is then to characterize the computational complexity of the CSP based upon properties of that set of relations. Feder and Vardi \cite{FederVardi1998} studied broad families of constraints which lead to a tractable CSP. Their work inspired what is known as the CSP Dichotomy Conjecture, postulating that every fixed set of constraint relations is either $\mathbf{NP}$-complete or tractable.

A discovery of Jeavons, Cohen, and Gyssens \cite{JeavonsCohenGyssens1997}, later refined by Bulatov, Jeavons and Krokhin \cite{BulatovJeavonsKrokhin2005} was the ability to translate the question of the complexity of the CSP over a set of relations to a question of algebra. Specifically, they showed that the complexity of any particular CSP depends solely on the \emph{polymorphisms} of the constraint relations, that is, the functions preserving all the constraints. The translation to universal algebra was made complete by Bulatov, Jeavons, and Krokhin in recognizing that to each CSP, one can associate an algebra whose operations consist of the polymorphisms of the constraints. Following this, the Dichotomy Conjecture of Feder and Vardi was recast as the \emph{Algebraic Dichotomy Conjecture}, a condition with a number of equivalent statements (summarized in \cite{BulatovValeriote}) which suggests a sharp dividing line between those CSPs that are  $\mathbf{NP}$-complete and those that are tractable, dependent solely upon universal algebraic conditions of the associated algebra. One of these conditions is the existence of a weak near-unanimity term (WNU, see Definition~\ref{def:WNU}). Roughly speaking, the Algebraic Dichotomy Conjecture asserts that an algebra corresponds to a tractable CSP if and only if it has a WNU term. The necessity of this condition was established in \cite{BulatovJeavonsKrokhin2005}. Our goal in this paper is to provide further evidence of sufficiency.

It follows easily from Definition~\ref{def:WNU} that a binary operation is weak near-unanimity if and only if it is commutative and idempotent. This motivates us to consider algebras with a single binary operation that is commutative and 
idempotent---CI-groupoids
for short. If the dichotomy conjecture is true, then every finite CI-groupoid should give rise to a tractable CSP.

In \cite{JeavonsCohenGyssens1997} it was proved that the dichotomy conjecture holds for CI-groupoids that are associative, in other words, for semilattices. This result was generalized in \cite{Bulatov2006} by weakening associativity to the identity $x(xy)\=xy$. In the present paper we continue this line of attack by considering several other identities that (in the presence of commutativity and idempotence) are strictly weaker than associativity. A family of such identities, those of Bol-Moufang type, is studied in Sections~\ref{sec:BMGroupoids} and~\ref{sec:Structure}. In Section~\ref{sec:OtherVarieties} we analyze CI-groupoids satisfying the self-distributive law $x(yz)\=(xy)(xz)$. In addition to proving that each of these conditions implies tractability, we establish some structure theorems that may be of further interest. The tractability results in this paper are  related to some unpublished work of Mar\'oti. On the whole, our results and his seem to be incomparable.

The early sections of the paper are devoted to supporting material. In Section~\ref{sec:Preliminaries}, we review the relevant concepts of universal algebra and constraint satisfaction. In Section~\ref{sec:Plonka} we discuss the P{\l}onka sum as well as a generalization which we will use as our primary structural tool. This is applied in Section~\ref{sec:Main} to obtain a general preservation result for tractable CSPs. We are hopeful that this technique will prove useful in future analysis of constraint satisfaction.

\section{Preliminaries}\label{sec:Preliminaries}
In order to achieve our main result, we must collect together several notions of the CSP (largely outlined in \cite{BulatovJeavons2001}), and ways of moving between them. We also survey the main algorithms at our disposal to establish the tractability of particular classes of CSPs.

\begin{defn}
An \emph{instance} of the CSP is a triple $\mathcal{R}=(V,A,\mathcal{C})$ in which:
\begin{itemize}
   \item $V$ is a finite set of \emph{variables},
   \item $A$ is a nonempty, finite set of \emph{values},
   \item $\mathcal{C}=\left\{(S_i,R_i)\mid i=1,\ldots,n\right\}$ is a set of \emph{constraints}, with each $S_i$ an $m_i$-tuple of variables, and each $R_i$ an $m_i$-ary relation over $A$ which indicates the allowed simultaneous values for variables in $S_i$.
\end{itemize}
Given an instance $\mathcal{R}$ of the CSP, we wish to answer the question: Does $\mathcal{R}$ have a \emph{solution}? That is, is there a map $f\colon V\rightarrow A$ such that for $1\leq i\leq n$, $f(S_i)\in R_i$?	
\end{defn}

The class of all CSPs is $\mathbf{NP}$-complete, but by restricting the form of relations allowed to appear in an instance, we can identify certain subclasses of the CSP which are tractable.

\begin{defn}
Let $\Gamma$ be a set of finitary relations over a set $A$. $\CSP(\Gamma)$ denotes the decision problem whose instances have set of values $A$ and with constraint relations coming from $\Gamma$.
\end{defn}

We refer to this first notion of the CSP as \emph{single-sorted}. A common example of the single-sorted $\CSP(\Gamma)$ is the graph $k$-colorability problem, given by $\Gamma = \{\neq_A\}$, where $\neq_A$ is the binary disequality relation on any set with $|A|=k$.

A second formulation of the CSP arises naturally in the context of conjunctive queries to relational databases (for more information about the connection see \cite[Definition 2.7] {BulatovJeavons2001}). For a class of sets $\mathcal{A}=\{A_i\mid i\in I\}$, a subset $R$ of $A_{i_1}\times\cdots\times A_{i_k}$ together with the list of indices $(i_1,\ldots,i_k)$ is called a \emph{$k$-ary relation over} $\mathcal{A}$  \emph{with signature} $(i_1,\ldots,i_k)$.

\begin{defn}
An \emph{instance} of the \emph{many-sorted} CSP is a quadruple $\mathcal{R}=(V,\mathcal{A},\delta,\mathcal{C})$ in which:
\begin{itemize}
   \item $V$ is a finite set of \emph{variables},
   \item $\mathcal{A}=\{A_i\mid i\in I\}$ is a collection of finite sets of \emph{values},
   \item $\delta\colon V\rightarrow I$ is called the \emph{domain function},
   \item $\mathcal{C}=\left\{(S_i,R_i)\mid i=1,\ldots,n\right\}$ is a set of \emph{constraints}. For $1\leq i\leq n$, $S_i=(v_1,\ldots,v_{m_i})$ is an $m_i$-tuple of variables, and each $R_i$ is an $m_i$-ary relation over $\mathcal{A}$ with signature $(\delta(v_1),\ldots,\delta(v_{m_i}))$ which indicates the allowed simultaneous values for variables in $S_i$.
\end{itemize}
Given an instance $\mathcal{R}$ of the many-sorted CSP, we wish to answer the question: Does $\mathcal{R}$ have a \emph{solution}? That is, is there a map $f\colon V\rightarrow \bigcup_{i\in I} A_i$ such that for each $v\in V$, $f(v)\in A_{\delta(v)}$, and for $1\leq i\leq n$, $f(S_i)\in R_i$?
\end{defn}

The single-sorted version of the CSP is obtained from the many-sorted by requiring the domain function $\delta$ to be constant. It is tacitly assumed that every instance of a constraint satisfaction problem can be encoded as a finite binary string. The length of that string is formally considered to be the size of the instance. We can restrict our attention to specific classes of the many-sorted CSP in a manner similar to the one we used in the single-sorted case. 

\begin{defn}\label{def:CSPgamma}
Let $\Gamma$ be a set of relations over the class of sets $\mathcal{A}=\{A_i\mid i\in I\}$. $\CSP(\Gamma)$ denotes the decision problem with instances of the form $(V,\mathcal{B},\delta,\mathcal{C})$ in which $\mathcal{B}\subseteq\mathcal{A}$ and every constraint relation is a member of $\Gamma$. 
\end{defn}

In either case (many- or single-sorted), we are concerned with determining which sets of relations result in a tractable decision problem.

\begin{defn}
Let $\Gamma$ be a set of relations. We say that $\Gamma$ is \emph{tractable} if for every finite subset $\Delta\subseteq\Gamma$, the class $\CSP(\Delta)$ lies in $\mathbf{P}$. If there is some finite $\Delta\subseteq\Gamma$ for which $\CSP(\Delta)$ is $\mathbf{NP}$-complete, we say that $\Gamma$ is \emph{$\mathbf{NP}$-complete}.
\end{defn}

Feder and Vardi \cite{FederVardi1998} conjectured that every finite set of relations is either tractable or $\mathbf{NP}$-complete, while it was Jeavons and his coauthors \cite{BulatovJeavons2001,BulatovJeavonsKrokhin2005, JeavonsStructure, JeavonsCohenGyssens1997} who made explicit the link between families of relations over finite sets and finite algebras that has made possible many partial solutions to the dichotomy conjecture.

An introduction to the necessary concepts from universal algebra (such as operation, relation, term, identity, and the operators $\opH$, $\opS$, $\opP$ and $\opV$) can be found in \cite{BergmanBook}, and we will follow the notation presented therein. In order to complete the transition from sets of relations to finite algebras, we collect a few more definitions.

\begin{defn} Let $A$ be a set, $\Gamma$ a set of finitary relations on $A$, $\mathcal{F}$ a set of finitary operations on $A$, $R$ an $n$-ary relation on $A$, and $f$ an $m$-ary operation on $A$.
\begin{enumerate}[ (1)] 
\item We say that $f$ is a \emph{polymorphism} of $R$, or that $R$ is \emph{invariant} under $f$ (see \cite[Definition~4.11]{BergmanBook}) if $$\overline{a}_1,\ldots,\overline{a}_m\in R\Rightarrow f(\overline{a}_1,\ldots,\overline{a}_m)\in R.$$
\item $\Pol(\Gamma)=\left\{f\mid f\text{ preserves every }R\in\Gamma\right\}$, the \emph{clone of polymorphisms} of $\Gamma$.
\item $\Inv(\mathcal{F})=\left\{R\mid R \text{ is invariant under every }f\in\mathcal{F}\right\}$, the \emph{relations invariant under} $\mathcal{F}$.
\item $\langle\Gamma\rangle$ denotes $\Inv(\Pol(\Gamma))$, the \emph{relational clone on $A$ generated by $\Gamma$}.
\end{enumerate}
\end{defn}

The following result (\cite[Corollary~2.17]{BulatovJeavonsKrokhin2005}) relates the computational complexity of a set of finitary relations to the complexity of the relational clone it generates.

\begin{thm}\label{thm:clonetract}
Let $\Gamma$ be a set of finitary relations on finite set $A$. $\Gamma$ is tractable if and only if $\langle\Gamma\rangle$ is tractable. If $\langle\Gamma\rangle$ is $\mathbf{NP}$-complete, then so is $\Gamma$.
\end{thm}

To every set of relations $\Gamma$ over a finite set $A$, we can associate the finite algebra  $\A_\Gamma=\langle A, \Pol(\Gamma)\rangle$. Likewise, to every finite algebra $\A=\langle A, \mathcal{F}\rangle$, we can associate the set of relations $\Inv(\mathcal{F})$. We call an algebra  $\A=\langle A, \mathcal{F}\rangle$ tractable ($\mathbf{NP}$-complete) precisely when $\Inv(\mathcal{F})$ is a tractable ($\mathbf{NP}$-complete) set of relations, and write $\CSP(\A)$ to denote the decision problem $\CSP(\Inv(\mathcal{F})$). In fact, combining Theorem~\ref{thm:clonetract} with the fact that $\langle \Gamma\rangle = \Inv(\Pol(\Gamma))$, the dichotomy conjecture can be settled by restricting one's attention to algebras.

For an individual algebra $\A=\langle A,\mathcal{F}\rangle$, the set $\Inv(\mathcal{F})$ of invariant relations on $A$ coincides with $\opSP_\text{fin}(\A)$, the set of subalgebras of finite powers of $\A$. We can extend this to the multisorted context as follows. Let $\{\A_i\mid i\in I\}$ be a family of finite algebras. By $\CSP(\{\A_i\mid i\in I\})$ we mean the many-sorted decision problem $\CSP(\Gamma)$ in which $\Gamma=\opSP_\text{fin}\{\A_i\mid i\in I\}$ as in Definition~\ref{def:CSPgamma}. Owing to the work of Bulatov and Jeavons, we can move between many-sorted CSPs and single-sorted CSPs while preserving 
tractability
by the following result \cite[Theorem~3.4]{BulatovJeavons2001}.

\begin{thm}\label{thm:multi2single}
Let $\Gamma$ be a set of relations over the finite sets $\{A_1,\ldots,A_n\}$. Then there exist finite algebras $\A_1,\ldots,\A_n$ with universes $A_1,\ldots,A_n$, respectively, such that the following are equivalent:
\begin{enumerate}[ \normalfont(a)] 
\item $\CSP(\Gamma)$ is tractable;
\item $\CSP(\{\A_1,\ldots,\A_n\})$ is tractable;
\item $\A_1\times\cdots\times\A_n$ is tractable.
\end{enumerate}
\end{thm}

A variety, $\class{V}$, of algebras is said to be \emph{tractable} if every finite algebra in $\class{V}$ is tractable. The tractability of many varieties has been established by identifying special term conditions.

\begin{defn}
For $k\geq 2$, a \emph{$k$-edge operation} on a set $A$ is a $(k+1)$-ary operation, $f$, on $A$ satisfying the $k$ identities:
\begin{align*}
f(x,x,y,y,y,\ldots,y,y)&\=y\\
f(x,y,x,y,y,\ldots,y,y)&\=y\\
f(y,y,y,x,y,\ldots,y,y)&\=y\\
f(y,y,y,y,x,\ldots,y,y)&\=y\\
&\hspace*{5pt}\vdots\\ 
f(y,y,y,y,y,\ldots,x,y)&\=y\\
f(y,y,y,y,y,\ldots,y,x)&\=y
\end{align*}
\end{defn}

\begin{defn}
A \emph{Maltsev operation} on a set $A$ is a ternary operation $q(x,y,z)$ satisfying $q(x,y,y)\=q(y,y,x)\=x$
\end{defn}

\begin{defn}\label{def:WNU}
A $k$-ary \emph{weak near-unanimity operation} on $A$ is an operation that satisfies the identities $$f(x,\ldots,x)\=x$$ and $$f(y,x,\ldots,x)\=f(x,y,\ldots,x)\=\cdots\=f(x,x,\ldots,x,y).$$ A $k$-ary \emph{near-unanimity operation} is a weak near-unanimity operation satisfying $f(y,x,\ldots,x)\=x$.
\end{defn}

An algebra is said to be \emph{congruence meet-semidistributive} ($\SD(\wedge)$) if its congruence lattice satisfies the implication
$$(x\wedge y\=x\wedge z)\Rightarrow (x\wedge(y\vee z)\=x\wedge y).$$
A variety $\class{V}$ is congruence meet-semidistributive if every algebra in $\class{V}$ is congruence meet-semidistributive. The existence of a strong Maltsev condition for congruence meet-semidistributivity was shown by Kozik, Krokhin, Valeriote and Willard.

\begin{thm}
[{\cite[Theorem~2.8]{KozikKrokhinValerioteWillard}}]\label{thm:SDMterms}
A locally finite variety is congruence meet-semidistributive if and only if it has $3$-ary and $4$-ary weak near-unanimity terms $v(x,y,z)$ and $w(x,y,z,u)$ that satisfy $v(y,x,x)\=w(y,x,x,x)$.
\end{thm}

Following from a result of Barto and Kozik, the existence of such terms $v$ and $w$ (which we call $\SD(\wedge)$ terms) is enough to establish the tractability of a variety.

\begin{thm}
 [{\cite[Theorem 3.7]{BartoKozik2009}}]\label{thm:SDMtractable}
If $\A$ is a finite algebra which lies in a congruence meet-semidistributive variety, then $\A$ is tractable.
\end{thm}

We can demonstrate the 
well-known
fact that the variety of (join) semilattices is $\SD(\wedge)$ (and hence tractable) by defining $v(x,y,z)=x\vee y \vee z$ and $w(x,y,z,u)=x\vee y\vee z\vee u$, and applying Theorem~\ref{thm:SDMterms}. A finite algebra which lies in a congruence meet-semidistributive variety gives rise to a Constraint Satisfaction Problem which is solvable by the so-called ``Local Consistency Method.'' Larose and Z{\'a}dori showed that every finite, idempotent algebra which gives rise to a CSP solvable by this same method must generate a congruence meet-semidistributive variety. The Barto and Kozik result shows the converse. 

The few subpowers algorithm, perhaps more widely known than the Local Consistency Method, is described by the authors in \cite{IMMVW} as the most robust ``Gaussian Like'' algorithm for tractable CSPs. It establishes the tractability of a finite algebra with a $k$-edge term, via the following result 
\cite[Corollary~4.2]{IMMVW}.

\begin{thm}\label{thm:BWtractable}
Any finite algebra which has has, for some $k\geq 2$, a $k$-edge term, is tractable. 
\end{thm}

Both Maltsev terms and near-unanimity terms give rise to $k$-edge terms, and thus the result of \cite{IMMVW} subsumes those of \cite{BulatovDalmau2006} and \cite{Dalmau-gmm}.

From the point of view of Universal Algebra, a \emph{quasigroup} is usually defined as an algebra $\langle A,\cdot,/, \backslash\rangle$ with three binary operations satisfying the identities

\begin{equation}\label{eq:qgpaxioms}
\begin{aligned}
 x\backslash(x\cdot y)&\= y, \qquad (x\cdot y)/y \= x,\\
 x\cdot(x\backslash y)&\= y, \qquad (x/y)\cdot y\= x.
\end{aligned}
\end{equation}

By a \emph{Latin square} we mean a groupoid $\langle A,\cdot\rangle$ such that for any $a,b\in A$, there are unique $c,d\in A$ such that $a\cdot c=b$ and $d\cdot a=b$. If $\langle A,\cdot,/, \backslash\rangle$ is a quasigroup, then $\langle A, \cdot\rangle$ is a Latin square. Conversely, every Latin square $\langle A,\cdot\rangle$ has an expansion to a quasigroup by defining $a\backslash b$ and $b/a$ to be the unique elements $c,d$ defined above.

The class of quasigroups forms a variety, axiomatized by \eqref{eq:qgpaxioms}. In fact, this variety has a Maltsev term, given by $q(x,y,z)=(x/(y\backslash y))\cdot (y\backslash z)$. It follows from Theorem~\ref{thm:BWtractable} that the variety of all quasigroups is tractable.

The situation for Latin squares is a bit more subtle. Neither a subgroupoid nor a homomorphic image of a Latin square is necessarily Latin. Thus the class of all Latin squares is not a variety. However, every \emph{finite} Latin square generates a variety that is term-equivalent to a variety of quasigroups. It follows that the term $q(x,y,z)$ given in the previous paragraph can be translated into a groupoid expression that will serve as a Maltsev term for this finitely generated variety. (The particular term obtained depends on the cardinality of the generating algebra.) Thus, from Theorem~\ref{thm:BWtractable}, we deduce that every finitely generated variety of Latin squares is tractable.

\section{P{\l}onka sums}\label{sec:Plonka}

A similarity type of algebras is said to be \emph{plural} if it contains no nullary operation symbols, and at least one non-unary operation symbol. Let $\mathcal{F}$ be a set of operation symbols, and $\rho\colon\mathcal{F}\rightarrow\mathbb{N}$ a plural similarity type. For any semilattice $\mathbf{S}=\langle S,\vee\rangle$, let $\mathbf{S}_\rho$ denote the algebra of type $\rho$ in which, for any $f\in\mathcal{F}$ with $\rho(f)=n$, $f(x_1,x_2,\ldots,x_n) = x_1\vee x_2\vee \cdots\vee x_n$. $\mathbf{S}$ can be recovered from $\mathbf{S}_\rho$ by taking, for any non-unary operation symbol $f$, $x\vee y=f(x,y,y,\ldots,y)$. The class $\class{Sl}_{\!\rho}=\{\mathbf{S}_\rho\mid \mathbf{S}\text{ a semilattice}\}$ forms a variety term-equivalent to the variety, $\class{Sl}$, of semilattices. Notice that when the similarity type consists of a single binary operation, $\class{Sl}_{\!\rho}$ and $\class{Sl}$ coincide.

An identity is called \emph{regular} if the same variables appear on both sides of the equals sign, and \emph{irregular} otherwise. A variety is called regular if it is defined by regular identities. In contrast, a variety is called \emph{strongly irregular} if it satisfies an identity $t(x,y)\=x$ for some binary term $t$ in which both $x$ and $y$ appear. Every strongly irregular variety has an equational base consisting of a set of regular identities and a single strongly irregular identity \cite{Melnik, Romanowska}. Note that most ``interesting'' varieties are strongly irregular---most Maltsev conditions involve a strongly irregular identity. For example, the Maltsev condition for congruence-permutability has a ternary term $q(x,y,z)$ satisfying $q(x,y,y)\=x$, which is a strongly irregular identity. By contrast, the variety of semilattices is regular.

The \emph{regularization},
$\widetilde{\vphantom{t}\smash{\class{V}\,}}\!$,
of a variety $\class{V}$ is the variety defined by all regular identities that hold in $\class{V}$. Equivalently, $\widetilde{\vphantom{t}\smash{\class{V}\,}}\!=\class{V}\vee\class{Sl}_{\!\rho}$, following from the fact that $\class{Sl}_{\!\rho}$ is the class of algebras satisfying all regular identities of type $\rho$. If $\class{V}$ is a strongly irregular variety, there is a very good structure theory for the regularization $\widetilde{\vphantom{t}\smash{\class{V}\,}}\!$ (due to P{\l}onka \cite{Plonka67, Plonka69}), which we shall now describe.

Recall that there are several equivalent ways to think of a semilattice: as an associative, commutative, idempotent groupoid $\langle S,\vee\rangle$; as a poset $\langle S,\leq_\vee\rangle$ with ordering $x\leq_\vee y \Leftrightarrow x\vee y = y$; and as the algebra $\mathbf{S}_\rho$ of type $\rho$ defined above.

\begin{defn}
Let $\langle S, \vee\rangle$ be a semilattice, $\{\A_s\mid s\in S\}$ a collection of algebras of plural type \mbox{$\rho\colon\mathcal{F}\rightarrow\mathbb{N}$}, and $\{\phi_{s,t}\colon\A_s\rightarrow\A_t \mid s\leq_\vee t\}$ a collection of homomorphisms satisfying $\phi_{s,s}=1_{A_s}$ and $\phi_{t,u}\circ\phi_{s,t}=\phi_{s,u}$. The \emph{P{\l}onka sum} of the system \mbox{$\langle \A_s : s\in S;\phi_{s,t}:s\leq_\vee t\rangle$} is the algebra $\A$ of type $\rho$ with universe $A=\bigudot \{A_s\mid s\in S\}$ and for $f\in\mathcal{F}$ a basic $n$-ary operation, $$f^\A(x_1,x_2,\ldots,x_n)=f^{\A_s}(\phi_{s_1,s}(x_1),\phi_{s_2,s}(x_2),\ldots,\phi_{s_n,s}(x_n))$$ in which $s=s_1\vee s_2\vee\cdots \vee s_n$ and $x_i\in A_{s_i}$ for $1\leq i\leq n$.
\end{defn}

In a P{\l}onka sum, the component algebras $\A_s$ (easily seen to be subalgebras of the P{\l}onka sum $\A$) are known as the \emph{P{\l}onka fibers}, while the homomorphisms between them are called the \emph{fiber maps}. The \emph{canonical projection} of a P{\l}onka sum $\A$ is the homomorphism \mbox{$\pi\colon\A\rightarrow\mathbf{S}_\rho$; $x\in A_s\mapsto s\in S$}, where $\mathbf{S}_\rho$ is the member of $\class{Sl}_{\!\rho}$ derived from $\mathbf{S}$. The algebra $\mathbf{S}_\rho$ is referred to as the \emph{semilattice replica} of the algebra $\A$, and the kernel of $\pi$ is the \emph{semilattice replica congruence}. Note that the congruence classes of this congruence are precisely the P{\l}onka fibers. In some cases, a very particular P{\l}onka sum will be useful.

\begin{defn}\label{def:Ainfty}
Let $\A$ be any algebra and $\mathbf{S}_2 = \langle \{0, 1\}, \leq_\vee \rangle$ the two-element join semilattice. We define the algebra $\A^\infty$ to be the P{\l}onka sum of the system  \mbox{$\langle \A_s : s\in S_2;\phi_{s,t}:s\leq_\vee t\rangle$}, where $\A_0=\A$, $\A_1$ is the trivial algebra of the same type as $\A$, and $\phi_{0,1}$ is the trivial homomorphism.
\end{defn}

A comprehensive treatment of P{\l}onka sums and more general constructions of algebras is presented in \cite{SmithRomanowskaModes}. We summarize just enough of the theory for our main result.

\begin{thm}[P{\l}onka's Theorem]\label{thm:plonkaequivalence}
Let $\class{V}$ be a strongly irregular variety of algebras of plural type $\rho$, defined by the set $\Sigma$ of regular identities, together with a strongly irregular identity of the form $x\vee y\=x$ (for some binary $\rho$-term $x\vee y$). Then the following classes of algebras coincide.

\begin{enumerate}[ \normalfont(1)] 
     \item The regularization, $\widetilde{\vphantom{t}\smash{\class{V}\,}}\!$, of $\class{V}$.
     \item The class $\clsop{P{\l}}(\class{V})$ of P{\l}onka sums of $\class{V}$-algebras. 
     \item \label{eq:plonkaidentities} The variety of algebras of type $\rho$ defined by the identities $\Sigma$ and the following identities (for $f\in\mathcal{F}$, $\rho(f)=n$):
          \begin{align}
                \label{eq:plonka1}\tag{P1}x\vee x &\= x\\
                \label{eq:plonka2}\tag{P2}(x\vee y)\vee z &\= x\vee(y\vee z)\\
                \label{eq:plonka3}\tag{P3}x\vee(y\vee z) &\= x\vee(z\vee y)\\
                \label{eq:plonka4}\tag{P4}y\vee f(x_1,x_2,\ldots,x_n) &\= y\vee x_1\vee x_2\vee\cdots\vee x_n\\
               \label{eq:plonka5}\tag{P5} f(x_1,x_2,\ldots,x_n)\vee y &\= f(x_1\vee y,x_2\vee y,\ldots,x_n\vee y)
          \end{align}
\end{enumerate}
\end{thm}

Note that in the variety $\class{V}$, the identities \eqref{eq:plonka1}--\eqref{eq:plonka5} defined in Theorem~\ref{thm:plonkaequivalence} are all direct consequences of $x\vee y\=x$. In $\widetilde{\vphantom{t}\smash{\class{V}\,}}\!$, $x\vee y$ is called the \emph{partition operation}, since it will decompose an algebra into the P{\l}onka sum of $\class{V}$-algebras as follows. For $\A\in\widetilde{\vphantom{t}\smash{\class{V}\,}}\!$, we define the relation \begin{equation}\label{eq:sigmadef}\sigma=\{(a,b)\colon a\vee b=a \text{ and }b\vee a=b\}.\end{equation} Clearly, $\sigma$ is both reflexive and symmetric. For transitivity, suppose that $a,b,c\in A$ are such that 
$a \mathrel{\sigma} b$ and $b \mathrel{\sigma} c$.
Then following from \eqref{eq:plonka2} and the definition of $\sigma$,
\begin{align*}
a\vee c=(a\vee b)\vee c &= a\vee(b\vee c) = a\vee b = a\\
c\vee a=(c\vee b)\vee a &= c\vee(b\vee a)=c\vee b=c
\end{align*}
Thus, $a\mathrel{\sigma} c$. Why is $\sigma$ a congruence on $\A$? Suppose that $a_1\mathrel{\sigma}b_1$, $\ldots$, $a_n\mathrel{\sigma}b_n$, and $f$ is a basic operation of $\A$. Then
\begin{align*}
f(a_1,\ldots, a_n)\vee f(b_1,\ldots, b_n) &\overset{\eqref{eq:plonka1}}{=}f(a_1,\ldots, a_n) \vee f(a_1,\ldots, a_n) \vee f(b_1,\ldots, b_n)\\
						     &\overset{\eqref{eq:plonka4}}{=}f(a_1,\ldots, a_n)\vee a_1\vee \cdots\vee a_n\vee b_1\vee\cdots\vee b_n\\
						     &\overset{\eqref{eq:plonka3}}{=}f(a_1,\ldots, a_n)\vee a_1\vee b_1\vee \cdots \vee a_n\vee b_n\\
						     &\overset{\sigma}{=} f(a_1,\ldots, a_n)\vee a_1 \vee \cdots\vee a_n\\
						     &\overset{\eqref{eq:plonka4}}{=}f(a_1,\ldots, a_n)\vee f(a_1,\ldots, a_n)\\
						     &\overset{\eqref{eq:plonka1}}{=}f(a_1,\ldots, a_n).
\end{align*}

Similarly, $f(b_1,\ldots, b_n)\vee f(a_1,\ldots, a_n)=f(b_1,\ldots, b_n)$, and so $\sigma$ is a congruence on $\A$. Each $\sigma$-class will be a $\class{V}$-algebra satisfying $x\vee y \= x$, and the quotient $\A/\sigma$ will be the algebra $\mathbf{S}_\rho$ for some semilattice $\mathbf{S}$. The algebra \A is the P{\l}onka sum over the semilattice $\A/\sigma$ of its $\sigma$-classes.

It turns out we do not need the full strength of Theorem~\ref{thm:plonkaequivalence} for our purposes. Let $\A$ be an algebra possessing a binary term $x\vee y$ satisfying \eqref{eq:plonka1}--\eqref{eq:plonka4}. Equation \eqref{eq:sigmadef} still defines a congruence $\sigma$ on $\A$ and $\A/\sigma$ is still a member of $\class{Sl}_{\!\rho}$. Such an algebra might not be a P{\l}onka sum, since we are no longer guaranteed the existence of fiber maps between congruence classes, defined in the proof of P{\l}onka's Theorem by $a/\sigma\rightarrow b/\sigma;x\mapsto x\vee b$. This is a homomorphism precisely when equation $\eqref{eq:plonka5}$ is satisfied.

\begin{defn}
We call a binary term $x\vee y$ satisfying the identities \eqref{eq:plonka1}--\eqref{eq:plonka4} in Theorem~\ref{thm:plonkaequivalence} a \emph{pseudopartition operation}.
\end{defn}

Let $x\vee y$ be a pseudopartition operation on $\A$. For any $n$-ary basic operation $f$ (and hence any term), we have $$f(x_1,\ldots,x_n)\in(x_1/\sigma\vee\cdots\vee x_n/\sigma)=(x_1\vee\cdots\vee x_n)/\sigma$$ as $$f(x_1,\ldots,x_n)\vee (x_1\vee \cdots \vee x_n)\=f(x_1,\ldots,x_n)\vee f(x_1,\ldots,x_n)\=f(x_1,\ldots,x_n)$$ and $$(x_1\vee\cdots\vee x_n)\vee f(x_1,\ldots,x_n)\=(x_1\vee\cdots\vee x_n)\vee(x_1\vee\cdots\vee x_n)\=(x_1\vee\cdots\vee x_n).$$ 
In particular, every $\sigma$-class is a subalgebra of $\A$.

\section{Main result}\label{sec:Main}

\begin{thm}\label{thm:pseudo-tractable}
Let $\A$ be a finite idempotent algebra with pseudopartition operation $x\vee y$, such that every block of its semilattice replica congruence lies in the same tractable variety. Then $\CSP(\A)$ is tractable.
\end{thm}
\begin{proof}
Let $\A$ be a finite idempotent algebra with pseudopartition operation $x\vee y$, and corresponding semilattice replica congruence $\sigma$. As we observed in the proof of Theorem~\ref{thm:plonkaequivalence}, each P{\l}onka fiber, $\A_a=a/\sigma$, for $a\in A$, is a subalgebra of $\A$.

Let $\mathcal{R}=(V,A,\mathcal{C}=\{( S_i,R_i)\mid i=1,\ldots,n\})$ be an instance of $\CSP(\A)$. We shall define an instance $$\mathcal{T}=(V,\{\A_a\mid a\in A\},\delta\colon V\rightarrow A;v\mapsto a_v,\mathcal{C}'=\{( S_i,T_i)\mid i=1,\ldots,n\})$$ of the multisorted $\CSP(\{\A_a\mid a\in A\})$, and reduce $\mathcal{R}$ to $\mathcal{T}$. By Theorem~\ref{thm:multi2single}, the tractability of $\CSP(\{\A_a\mid a\in A\})$ is equivalent to the tractability of $\CSP(\prod_{a\in A}\A_a)$. Since the product $\prod_{a\in A}\A_a$ is assumed to lie in a tractable variety, if we can reduce $\mathcal{R}$ to $\mathcal{T}$, then our original problem, $\CSP(\A)$, will be tractable. 

First, we define the missing pieces of the instance $\mathcal{T}$. Let $1\leq i\leq n$. Then $S_i$ has the form $(v_1,\ldots,v_{m_i})$, where each $v_j$ is an element of $V$. For a variable $v\in V$, we shall write $v\in S_i$ to indicate that $v=v_j$ for some $j\leq {m_i}$. Moreover, when this occurs, $\pi_v(R_i)$ will denote the projection of $R_i$ onto the $j^\text{th}$ coordinate.

For $v\in V$, define $J_v=\{i\leq n\mid v\in S_i\}$ and set $$B_v=\bigcap_{i\in J_v}\pi_v(R_i).$$
Since each $R_i$ is an invariant relation on $\A$, $B_v$ is a subuniverse of $\A$. It is easy to see that if $f$ is a solution to $\mathcal{R}$ then $f(v)\in B_v$. Consequently, we can assume without loss of generality that each $R_i$ is a subdirect product of $\prod_{v\in S_i}B_v$.

We define the element $a_v=\bigvee B_v$, applying the term $\vee$ to take the join of the entire set $B_v$.  In principle, the order matters (since we are not assuming that $\vee$ is commutative), however as a consequence of the definition of a pseudopartition operation, the result will always be in the same $\sigma$-class regardless of order. We define $B'_v=\A_{a_v}=a_v/\sigma$. Since $B_v\leq \A$, we have that $a_v\in B_v\cap B'_v$. For $i=1,\ldots,n$, with $S_i=(v_1,\ldots,v_{m_i})$, define $T_i=R_i\cap 
\big(B'_{v_1}\times\cdots\times B'_{v_{m_i}} \big)$.

Obviously, any solution to $\mathcal{T}$ is a solution to $\mathcal{R}$. We now show that any solution to $\mathcal{R}$ can be transformed into a solution to $\mathcal{T}$. Let $f\colon V\rightarrow A$ be a solution to $\mathcal{R}$, and define $$g\colon V\rightarrow\bigcup_{a\in A} A_a;v\mapsto f(v)\vee a_v.$$
We need to show that $g(S_i)\in T_i$ and $g(v)\in A_{a_v}=a_v/\sigma$. We first claim that $$\left(\forall v\in V\text{ and } b\in B_v\right)\,\,b\vee a_v\in a_v/\sigma.$$ To see this, observe that
\begin{equation}\label{eq:aveesigma}
a_v\vee(b\vee a_v)= a_v\vee b\vee \bigvee B_v=a_v\vee\bigvee B_v=a_v\vee a_v=a_v
\end{equation}
and $$(b\vee a_v)\vee a_v=b\vee(a_v\vee a_v)=b\vee a_v.$$ That $b\vee a_v\equiv a_v\pmod\sigma$ now follows from \eqref{eq:sigmadef}. Since $f$ is a solution to $\mathcal{R}$, for any $v\in V$, $f(v)\in B_v$. From \eqref{eq:aveesigma}, with $b=f(v)$, we obtain $g(v)=f(v)\vee a_v\in B'_v=A_{a_v}$. 

Fix an index $i\leq n$. Since each $R_i$ is a subdirect product, for every $v\in S_i$ there is a tuple $\mathbf{r}^v\in R_i$ with $\pi_v(\mathbf{r}^v)=a_v$. Furthermore, for each $v\in S_i$,
\begin{align*}
\pi_v(g(S_i))=g(v) &= f(v)\vee a_v\\
			     &= f(v)\vee \bigvee B_v\\
			     &\overset{*}{=} f(v)\vee \bigvee B_v\vee\bigvee_{\substack{w\neq v \\w\in S_i}}\pi_v(\mathbf{r}^w)\\
			     &=f(v)\vee a_v\vee\bigvee_{\substack{w\neq v \\w\in S_i}}\pi_v(\mathbf{r}^w)\\
			     &= f(v)\vee\bigvee_{w\in S_i}\pi_v(\mathbf{r}^w).
\end{align*}
The starred equality follows from \eqref{eq:plonka1}--\eqref{eq:plonka3} and $\pi_v(\mathbf{r^w})\in B_v$. The above allows us to conclude that $g(S_i)=f(S_i)\vee\bigvee_{w\in S_i}\mathbf{r}^w\in R_i\cap \prod_{v\in S_i} B'_v=T_i$, so $g$ is a solution to $\mathcal{T}$, which completes the proof.
\end{proof}

\section{Bol-Moufang groupoids}\label{sec:BMGroupoids}
\subsection{Definitions}
We call $\B=\langle B,\cdot\rangle$ a \emph{CI-groupoid} if ``$\cdot$'' is a commutative and idempotent binary operation. Typically, we will omit the $\cdot$ and indicate multiplication in a groupoid by juxtaposition. The associative law is, of course, the identity $x(yz) \approx (xy)z$. Identities weaker than associativity have been studied in several contexts, most notably in quasigroup theory. Indeed, quasigroups are typically thought of as a nonassociative generalization of groups. Several of these identities are important enough to have earned names of their own, such as the flexible law $x(yx) \approx (xy)x$ and the Moufang law $(x(yz))x \approx (xy)(zx)$. Moufang's work goes back to 1935, when she showed that several such identities are all equivalent relative to the variety of loops (i.e., quasigroups with identity). 

The first systematic study of the implications among  weak associative laws seems to be \cite{Fenyves1969}. That paper enumerated 60 weak associative laws in 3 variables, with one variable repeated. Since that set included the Moufang law and another well-known identity due to Bol, Fenyves called these ``identities of Bol-Moufang type.'' Additional analysis of the relationship among these identities appears in \cite{Kunen1996, PhillipsVojtechovskyL2005, PhillipsVojtechovskyQ2005}. 

In this section we continue the study of weak associative laws. However, instead of working in the context of quasigroups and loops, we work within the variety of commutative, idempotent groupoids. Let $\class{C}$ stand for the variety of all CI-groupoids. A groupoid identity $p\=q$ is of \emph{Bol-Moufang type} if:

\begin{enumerate}[ (i)] 
 \item the same 3 variables appear in $p$ and $q$,
 \item one of the variables appears twice in both $p$ and $q$,
 \item the remaining two variables appear once in each of $p$ and $q$,
 \item the variables appear in the same order in $p$ and $q$.
 \end{enumerate}

One example is the Moufang law $x(y(zy))\=((xy)z)y$. There are 60 such identities, and a systematic notation for them was introduced in \cite{PhillipsVojtechovskyL2005, PhillipsVojtechovskyQ2005}. A variety of CI-groupoids is said to be of \emph{Bol-Moufang type} if it is defined by one additional identity which is of Bol-Moufang type. We say that two identities are \emph{equivalent} if they determine the same subvariety, relative to some underlying variety. In the present section, this will be the variety $\class{C}$ of CI-groupoids. Phillips and Vojt{\v{e}}chovsk{\'y} studied the equivalence of Bol-Moufang identities relative to the varieties of loops and quasigroups, requiring the binary operation appearing in a Bol-Moufang identity to be the underlying multiplication.

Let $p\=q$ be an identity of Bol-Moufang type with $x$, $y$, and $z$ the only variables appearing in $p$ and $q$. Since the variables must appear in the same order in $p$ and $q$, we can assume without loss of generality that they are alphabetical in order of first occurrence. There are exactly 6 ways in which the $x$, $y$, and $z$ can form a word of length 4 of this form, and there are exactly 5 ways in which a word of length 4 can be bracketed, namely:
\begin{center}
\begin{tabular}{c|cp{0.5in}c|c}
$A$ & $xxyz$ & & $1$ & $o(o(oo))$\\
$B$ & $xyxz$ & & $2$ & $o((oo)o)$\\
$C$ & $xyyz$ & & $3$ & $(oo)(oo)$\\
$D$ & $xyzx$ & & $4$ & $(o(oo))o$\\
$E$ & $xyzy$ & & $5$ &$((oo)o)o$ \\
$F$ & $xyzz$ & 
\end{tabular}
\end{center}

If $X$ is one of $A$, $B$, $C$, $D$, $E$ or $F$, and $1\leq i< j$, let $Xij$ be the identity whose variables are ordered according to $X$, whose left-hand side is bracketed according to $i$, and whose right-hand side is bracketed according to~$j$. For instance, \id{E15}{\text{i.e., 
}x(y(zy))}{((xy)z)y} is (one version of) the Moufang law. Following from our previous remarks, any identity of Bol-Moufang type can be transformed into some identity $Xij$ by renaming the variables and possibly interchanging the left- and right-hand sides. There are therefore $6 \cdot (4 + 3 + 2 + 1) = 60$ distinct nontrivial identities of Bol-Moufang type. 

Define the operation $\cdot^\text{op}$ by $x\cdot^\text{op}y=y\cdot x$. The 
\emph{dual} 
$p'$ of a groupoid term $p$ is the result of replacing all occurrences of $\cdot$ in $p$ with $\cdot^\text{op}$. The dual of a groupoid identity $p\=q$ is the identity $q'\=p'$. This notion of duality is consistent with the one given in \cite{PhillipsVojtechovskyL2005}. As an example, the dual of the Moufang law $x(y(zy))\=((xy)z)y$ is the identity $y(z(yx))\=((yz)y)x$. By renaming variables, we can rewrite this as $x(y(xz))\=((xy)x)z$, identified as $B15$ using the systematic notation above. One can easily identify the dual of any identity $Xij$ of Bol-Moufang type with the identity $X'j'i'$ of Bol-Moufang type computed by the rules:
$$A'=F,~~B'=E,~~C'=C,~~D'=D,~~1'=5,~~2'=4,~~3'=3.$$
We will indicate the dual of $Xij$ by $(Xij)'$, and call an identity $Xij$ of Bol-Moufang type \emph{self-dual} if $Xij$ and $(Xij)'$ are equal. For any ordering $X$ or parenthesization $i$, $X''=X$ and $i''=i$ always.

In the following subsections we explore the varieties of CI-groupoids of Bol-Moufang type. The analysis consists of a mix of equational derivation, display of counterexamples, and application of Maltsev conditions. This work was greatly aided by two software packages: Prover9/Mace4 \cite{prover9-mace4} and the Universal Algebra Calculator \cite{UACalc}. 

Most of the implications among the equations were first discovered using Prover9. However, this software produces derivations that are only barely human-readable. We found that it took considerable effort to rewrite the proofs to be accessible to an average reader.
Many of the equational derivations are quite long and are collected into an appendix. To save on printing costs, the appendix is not included in the published version of this paper. The entire paper, including appendix, is available online at \url{http://www.arxiv.org}, or \url{http://orion.math.iastate.edu/cbergman/manuscripts/cigcsp.pdf}.

Examples were produced by Mace4. As a rule it is a simple matter to read the Cayley table for a binary operation and verify the witnesses to an inequation. Finally, the Universal Algebra Calculator was very useful for computing congruences and searching for Maltsev conditions that hold in a finite algebra.

\subsection{Equivalences}
Before we can classify the complexity of the CSP corresponding to varieties of CI-groupoids of Bol-Moufang type, it will be necessary to determine which of the identities are equivalent. After determining the distinct varieties, we will establish the tractability of several using known tools. A summary of the equivalences is given in Table~\ref{tbl:equiv}. We begin with an observation that will shorten the proofs considerably.

\begin{table}[ht]
\caption{Varieties of CI-groupoids of Bol-Moufang type.} 
\centering
\begin{tabular}{| c | p{3.5in} |}
\hline
\bf{Name} & \bf{Equivalent Identities}\\ \hline
$\class{C}$ & $B45$, $D24$, $E12$\\ \hline
$\class{2Sl}$ & $A13$, $A45$, $C12$, $C45$, $F12$, $F35$\\ \hline
$\class{X}$ & $A24$, $A25$, $B24$, $B25$, $E14$, $E24$, $F14$, $F24$\\ \hline
$\class{Sl}$ & $A12$, $A15$, $A23$, $A34$, $A35$, $B14$, $B15$, $B34$, $B35$, $C13$, $C14$, $C23$, $C24$, $C25$, $C34$, $C35$,
   $D12$, $D14$, $D23$, $D25$, $D34$, $D45$, $E13$, $E15$, $E23$, $E25$, $F13$, $F15$, $F23$, $F34$, $F45$\\ \hline
$\class{T}_2$ & $C15$\\ \hline
$\class{T}_1$ & $A14$, $F25$\\ \hline
$\class{S}_2$ & $B12$, $D15$, $E45$\\ \hline
$\class{S}_1$ & $B13$, $B23$, $D13$, $D35$, $E34$, $E35$\\ \hline
\end{tabular}
\label{tbl:equiv}
\end{table}


\begin{rmk}\label{rmk:equivdual} For commutative groupoids, each identity of Bol-Moufang type is equivalent to its dual. In fact, for any term $p$ in a commutative groupoid, $p'\approx p$ holds.
\end{rmk}

\begin{thm}
The Bol-Moufang identities $A14$ and $F25$ are equivalent, defining the variety we call $\class{T}_1$.
\end{thm}
\begin{proof}
Follows immediately since $F25=(A14)'$.
\end{proof}

Remarkably, $C15$ is not equivalent to any other identity of Bol-Moufang type.

\begin{thm}
The identity $C15$ is self-dual, and defines the variety we call~$\class{T}_2$.
\end{thm}

Many of the below equivalences follow without the use of all of our assumptions, which may be worth investigating further. An additional remark justifies the study of Bol-Moufang identities as generalizations of associativity, and will prove useful in a few of the theorems.

\begin{rmk}\label{rmk:BMassoc} In any groupoid, associativity implies each of the identities of Bol-Moufang type.
\end{rmk}

\begin{thm} The following Bol-Moufang identities are pairwise equivalent, and determine the variety $\class{S}_1$: $B13$, $B23$, $D13$, $D35$, $E34$, $E35$.
\end{thm}
\begin{proof}
Identities 
$B13$ and $D13$ are equivalent by commuting the last two variables. To see that $B13$ and $B23$ are equivalent, interchange the roles of $y$ and~$z$, and apply commutativity. The remaining three identities are dual to the others.
\end{proof}

\begin{thm} The following Bol-Moufang identities are pairwise equivalent, and determine the variety $\class{S}_2$: $B12$, $D15$, $E45$.
\end{thm}
\begin{proof}
\id{B12}{x(y(xz))}{x((yx)z)} and \id{D15}{x(y(zx))}{((xy)z)x} are equivalent under commutativity alone.  $D15$ is self-dual, while $E45$ is the dual of $B12$.
\end{proof}

In \cite{Bulatov2006}, Bulatov proved the tractability of the variety of \emph{2-semilattices}, those groupoids satisfying all two-variable semilattice identities. In particular, this class is axiomatized by commutativity, idempotence, and the \emph{2-semilattice law}: $x(xy)\=xy$.

\begin{thm} \label{thm:bm2sml}The following Bol-Moufang identities are equivalent to the 2-semilattice law, and determine the variety $\class{2SL}$: $A13$, $A45$, $C12$, $C45$, $F12$, $F35$.
\end{thm}
\begin{proof}
The 2-semilattice law, together with idempotence, implies each of the listed identities. To see how the 2-semilattice law follows from the given identities, a few easy observations are all that is needed. For \id{A13}{x(x(yz))}{(xx)(yz)}, replace $z$ with $y$ and complete the derivation using idempotence. For \id{A45}{(x(xy))z}{((xx)y)z}:
\begin{align*}
x(xy) &\= (x(xy))(x(xy)) \=  ((xx)y)(x(xy))\\
&\= (xy)(x(xy)) \= (x(xy))(xy)\\
&\=((xx)y)(xy) \= (xy)(xy)\=(xy).
\end{align*}
For \id{C12}{x(y(yz))}{x((yy)z)}:
\begin{align*}
x(xy) &\= (x(xy))(x(xy)) \= (x(xy))((xx)y)\\
&\= (x(xy))(xy) \= (xy)(x(xy))\\
&\= (xy)((xx)y)\=(xy)(xy)\=(xy).
\end{align*}

The remainder of the identities are dual to those investigated, so it follows from Remark~\ref{rmk:equivdual} that they each imply the 2-semilattice law.
\end{proof}

The following lemmas will aid in proving the largest groups of equivalences.

\begin{lem}\label{lem:I2}
Each of following Bol-Moufang identities, together with idempotence, implies the 2-semilattice law: $A24$, $A25$, $A34$, $B35$, $C35$, $D23$.
\end{lem}
\begin{proof}
For \id{A24}{x((xy)z)}{(x(xy))z}: $$x(xy)\=x((xx)y)\=(x(xx))y\=(xx)y\=xy.$$
For \id{A25}{x((xy)z)}{((xx)y)z}: $$x(xy)\=x((xy)(xy))\=((xx)y)(xy)\=(xy)(xy)\=xy.$$
For \id{A34}{(xx)(yz)}{(x(xy))z}: $$x(xy)\=(xx)(xy)\=(x(xx))y\=xy.$$
For \id{B35}{(xy)(xz)}{((xy)x)z} and \id{C35}{(xy)(yz)}{((xy)y)z}: $$x(xy)\=(xx)(xy)\=((xx)x)y\=xy.$$
For \id{D23}{x((yz)x)}{(xy)(zx)}: See Appendix.
\end{proof}

\begin{lem}\label{lem:IC2}
Each of the following Bol-Moufang identities, together with commutativity and idempotence, implies the 2-semilattice law: $A15$, $A23$, $B14$, $C14$.
\end{lem}
\begin{proof}

For \id{A15}{x(x(yz))}{((xx)y)z}:
\begin{align*}
x(xy)&\=(xy)x\=((xx)y)x\=x(x(yx))\=x(x(xy))\=x(x(x(yy)))\\
&\=x(((xx)y)y)\=x((xy)y)\=((yx)y)x\=(((yx)(yx))y)x\\
&\=(yx)((yx)(yx))\=yx\=xy.
\end{align*}
For \id{A23}{x((xy)z)}{(xx)(yz)}: $$x(xy)\=x((xy)(xy))\=(xx)(y(xy))\=x(y(xy))\=x((xy)y)\=(xx)(yy)\=xy.$$
For \id{B14}{x(y(xz))}{(x(yx))z}: $$x(xy)\=x(yx)\=x(y(xx))\=(x(yx))x\=x(x(xy))\=(x(xx))y\=xy.$$
For \id{C14}{x(y(yz))}{(x(yy))z}: 
\[
x(xy)\=(yx)x\=(y(xx))x\=y(x(xx))\=yx\=xy.\qedhere
\]
\end{proof}

\begin{figure}[h] 
  \centering
    \begin{center}
      \begin{tabular}{c|ccc}
        &0&1&2\\
        \hline
        0&0&2&1\\
        1&0&1&2\\
        2&0&1&2
      \end{tabular}\\[3pt]
    \end{center}  \caption{Table for Example~\ref{ex:tab1a}}
    \label{fig:tab1}
\end{figure}

\begin{ex}\label{ex:tab1a}
Figure~\ref{fig:tab1} is an idempotent groupoid satisfying $A15$ and $A23$ which does not satisfy the 2-semilattice law (it fails since $0(0\cdot 1)\neq 0\cdot 1$).
\end{ex}

\begin{lem}\label{lem:F45-IC2}
$F45$, together with commutativity and idempotence, implies the 2-semilattice law.
\end{lem}
\begin{proof}
\id{F45}{(x(yz))z}{((xy)z)z} commutes to become $z((xy)z)\=z(x(yz))$. A few intermediate identities:
\begin{enumerate}[ (1)] 
\item\label{lem:F45-1} $(xy)(x(y(xy)))\=xy$ follows by replacing $z$ with $xy$ in the commuted version of $F45$.
\item $(yx)x\=x(y(yx))$ follows by replacing $x$ with $y$, and $z$ with $x$ in the commuted $F45$.
\item\label{lem:F45-3} $x(yx)\=x(y(xy))$ is just the previous identity with commutativity applied.
\item $(xy)(x(yx))\=xy$ follows from \eqref{lem:F45-1} and \eqref{lem:F45-3} above.
\end{enumerate}

We now have enough for the 2-semilattice law:

\begin{align*}
x(xy) & \= x(yx) \= [x(yx)][x(yx)]\\
& \= [x(yx)][x(y(xy))]\\
& \= [x(yx)][x(y(x(yx)))]\\
& \= [x(yx)][xy] \= [xy][x(yx)]\=xy.\qedhere
\end{align*}
\end{proof}

Several of the identities in Lemmas~\ref{lem:I2} and~\ref{lem:IC2} determine a subvariety of $\class{C}$ consisting of 2-semilattices. However, as nothing further was known about this subvariety as of this writing, we give it the name $\class{X}$.

\begin{thm} The following Bol-Moufang identities are pairwise equivalent, relative to the variety $\class{C}$ of commutative idempotent groupoids, and determine the variety $\class{X}$, a subvariety of 2-semilattices: $A24$, $A25$, $B24$, $B25$, $E14$, $E24$, $F14$, $F24$.
\end{thm}
\begin{proof}
The identities $A24$ and $B24$ are easily seen to be equivalent by commuting the variables in the innermost set of parentheses. We will show that $A24$ and $A25$ are equivalent, with the help of Lemma~\ref{lem:I2}. To see that $A24$ implies $A25$, observe that $((xx)y)z\=(xy)z\=(x(xy))z\=x((xy)z)$. Conversely, from $A25$ we can derive $x((xy)z)\=((xx)y)z\=(xy)z\=(x(xy))z$. Using the fact that $A24$ and $A25$ are equivalent, we show that $A25$ and $B25$ are equivalent. Assuming $A25$ (from which the 2-semilattice law follows by Lemma~\ref{lem:I2}):
\begin{align*}
(x(yx))z & \= x((xy)z)\\
& \= ((xx)y)z\\
& \= (xy)z\\
& \= (x(xy))z \=((xy)x)z,
\end{align*}
which is $B25$. Assuming $B25$, we show $A24$ as follows:
\begin{align*}
(x(xy))z & \= x((yx)z)\\
& \= ((xy)x)z\\
& \= (x(xy))z.
\end{align*}
The remaining identities are dual to those investigated.
\end{proof}

\begin{thm}\label{thm:bmsl} Each of the following Bol-Moufang identities is equivalent to associativity, and determines the variety $\class{SL}$ of semilattices:  $A12$, $A15$, $A23$, $A34$, $A35$, $B14$, $B15$, $B34$, $B35$, $C13$, $C14$, $C23$, $C24$, $C25$, $C34$, $C35$, $D12$, $D14$, $D23$, $D25$, $D34$, $D45$, $E13$, $E15$, $E23$, $E25$, $F13$, $F15$, $F23$, $F34$, $F45$.
\end{thm}
\begin{proof}
We proceed via a few closed loops of equivalences. Wherever the 2-semilattice law is used, it has already been proven to hold in Lemma~\ref{lem:I2}, Lemma~\ref{lem:IC2}, or Lemma~\ref{lem:F45-IC2}. Associativity implies any of the listed identities by our previous remark.

\begin{itemize}
\item $A23\Rightarrow D12\Rightarrow D14 \Rightarrow F45\Rightarrow F34\Rightarrow A23$
  \begin{itemize}
  \item $A23\Rightarrow D12$: $$x(y(zx))\=x((xz)y)\=(xx)(zy)\=x(zy)\=x(x(zy))\=x((yz)x)$$
  \item $D12$ and $D14$ are equivalent under commutativity.
  \item $D12\Rightarrow F45$: $$(x(yz))z\=z(x(yz))\=z((xy)z)\=((xy)z)z$$
  \item $F45\Rightarrow F34$: $$(xy)(zz)\=(xy)z\=((xy)z)z$$
  \item $F34$ is the dual of $A23$.
  \end{itemize}

\item $A23\Rightarrow C35\Rightarrow C34\Rightarrow$ Associativity $\Rightarrow A34\Rightarrow$ Associativity $\Rightarrow A23$
  \begin{itemize}
  \item $A23\Rightarrow C35$: \begin{align*}(xy)(yz)&\=[(xy)(xy)](yz)\=(xy)[((xy)y)z]\\&\=(xy)((xy)z)\=(xy)z\=((xy)y)z\end{align*}
  \item $C35\Rightarrow C34$: $$(xy)(yz)\=((xy)z)z\=(xy)z\=(x(yy))z$$
  \item $C34\Rightarrow$ Associativity: $$(xy)z\=(x(yy))z\=(xy)(yz)\=(zy)(yx)\=(z(yy))x\=(zy)x\=x(yz)$$
  \item $A34\Rightarrow$ Associativity: $$x(yz)\=(xx)(yz)\=(x(xy))z\=(xy)z$$
  \end{itemize}

\item $C35\Rightarrow B35\Rightarrow D23\Rightarrow C14\Rightarrow A15\Rightarrow C34$
  \begin{itemize}
  \item $C35\Rightarrow B35$: $$(xy)(xz)\=(yx)(xz)\=((yx)x)z\=((xy)x)z$$
  \item $B35\Rightarrow D23$: \begin{align*}x((yz)x)&\=x(yz)\=(yz)x\=(yz)(yx)\=(yx)(yz)\=((yx)y)z\\&\=(yx)z\=(xy)z\=((xy)x)z\=(xy)(xz)\=(xy)(zx)\end{align*}
  \item $D23\Rightarrow C14$: \begin{align*}x(y(yz))&\=x(yz)\=x((yz)x)\=(xy)(zx)\\&\=(yx)(xz)\=(yx)[(xz)(yx)]\=[(yx)x][z(yx)]\\&\=[yx][z(yx)]\=z(yx)\=(xy)z\=(x(yy))z\end{align*}
  \item $C14\Rightarrow A15$: $$x(x(yz))\=x(yz)\=x(y(yz))\=(x(yy))z\=(xy)z\=((xx)y)z$$
  \item $A15\Rightarrow C34$: \begin{align*}(xy)(yz)&\=(xy)((xy)(yz))\=(((xy)(xy))y)z\=((xy)y)z\=(xy)z\=(x(yy))z\end{align*}
  \end{itemize}
  
\item $B35\Leftrightarrow B14 \Leftrightarrow B15$
  \begin{itemize}
  \item $B35\Rightarrow B14$:
	\begin{enumerate}
	\item $B35$ simplifies to $(xy)(xz)\=(xy)z$ under the 2-semilattice law.
	\item $(xy)z\=(xz)y$ follows by permuting the variables in the left hand side of the above.
	\item $x(yz)\=z(xy)$ follows by permuting the variables in the above, and applying commutativity.
	\item Lastly, using the previous equation with $xz$ substituted for $z$ yields $x(y(xz))\=(xz)(xy)=(xy)z\=(x(xy))z\=(x(yx))z$, which is $B14$.
	\end{enumerate}
  \item $B14\Rightarrow B35$: \begin{align*}(xy)(xz)&\=(yx)(xz)\=(x(yx))(xz)\=x(y(x(xz)))\\&\=x(y(xz))\=(x(yx))z\=((xy)x)z\end{align*}
  \item $B14$ and $B15$ are equivalent under commutativity.
  \end{itemize}
  \end{itemize}

For the remaining identities: Applying idempotence, one can derive associativity from \id{A35}{(xx)(yz)}{((xx)y)z} or \id{C24}{x((yy)z)}{(x(yy))z}, and so both are equivalent to associativity. \id{B34}{(xy)(xz)}{(x(yx))z} and \id{B35}{(xy)(xz)}{((xy)x)z} are equivalent under commutativity. The remaining identities are dual to those investigated.
\end{proof}

There is one last class of equivalent identities of Bol-Moufang type. It is in some sense trivial.

\begin{thm}
The identities \id{B45}{(x(yx))z}{((xy)x)z}, \id{D24}{x((yz)x)}{(x(yz))x}, and \id{E12}{x(y(zy))}{x((yz)y)} are equivalent, and determine the variety $\class{C}$.
\end{thm}
\begin{proof}
It is easy to see that all three identities follow immediately from commutativity. 
\end{proof}

It is worth noting that although any one of $B45$, $D24$, or $E12$ are immediate consequences of commutativity, the reverse implications are false, even in the presence of idempotence.

\begin{ex}
A two element left-zero semigroup satisfies $B45$, $D24$, and $E12$, but is not commutative. 
\end{ex}

\subsection{Implications}
We now show how the 8 varieties of CI-groupoids of Bol-Moufang type are related.

\begin{thm}
The following inclusions hold among the varieties of CI-groupoids of Bol-Moufang type: $\class{SL} \subseteq \class{X} \subseteq \class{2SL} \subseteq \class{C}$, $\class{SL}\subseteq \class{T}_1 \subseteq \class{T}_2 \subseteq \class{C}$, $\class{SL} \subseteq \class{S}_1 \subseteq \class{S}_2 \subseteq \class{C}$.
\end{thm}
\begin{proof}
The variety $\class{SL}$ of semilattices is contained in all the others, following from Remark~\ref{rmk:BMassoc}. Likewise, they are all trivially contained in $\class{C}$. To see that $\class{X}$ is contained in $\class{2SL}$, note that in the proof of Lemma~\ref{lem:I2}, we showed that both $A24$ and $A25$, which define the variety $\class{X}$, imply the 2-semilattice law. To see that $\class{T}_1 \subseteq \class{T}_2$, we show that \id{A14}{x(x(yz))}{(x(xy))z} implies \id{C15}{x(y(yz))}{((xy)y)z}. Assuming $A14$, we have: $x(y(yz))\=(y(yz))x\=y(y(zx))\=y(y(xz))\=(y(yx))z\=((xy)y)z$. Lastly, to see that $\class{S}_1 \subseteq \class{S}_2$, we show that \id{B13}{x(y(xz))}{(xy)(xz)} implies \id{B12}{x(y(xz))}{x((yx)z)}. Assuming $B13$, we have $x(y(xz))\=(xy)(xz)\=(xz)(xy)\=x(z(xy))\=x((yx)z)$.
\end{proof}

\begin{figure}[ht] 
\begin{center}
\includegraphics[scale=0.5]{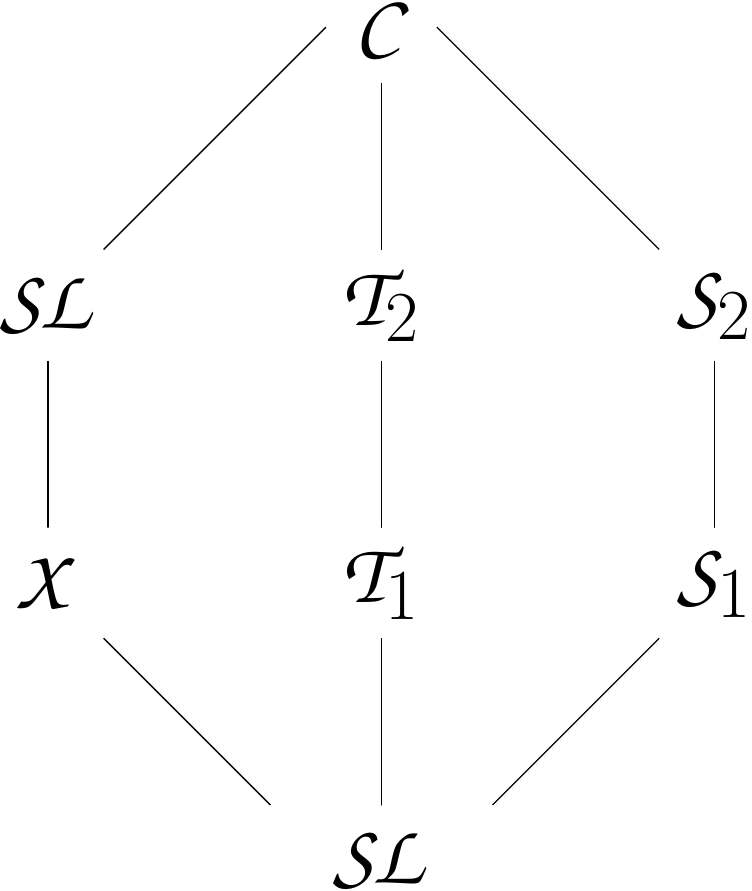}
\end{center}
\caption{Varieties of CI-groupoids of Bol-Moufang type}
\label{fig:bmhasse}
\end{figure}

A Hasse diagram of the situation (with inclusions directed upward, so that higher varieties are larger) is shown in Figure~\ref{fig:bmhasse}. Up to this point, we have justified only the inclusions, but we must still show that they are \emph{proper}, and that no inclusions have been missed.

\subsection{Distinguishing examples}\label{sec:distinguishing}

We now show that the 8 varieties of CI-groupoids of Bol-Moufang type are distinct. We have aimed to use as few examples as possible. While the 7 groupoids presented suffice to show that all inclusions are proper, there may be some larger groupoids which subsume multiple examples. For readability, and since each example is commutative, only the upper triangle of each Cayley table is given.

\begin{ex}\label{ex:tab2a} Figure~\ref{fig:tab2}(a) is a CI-groupoid which is not in $\class{2SL}\cup\class{T}_2\cup\class{S}_2$. The 2-semilattice law fails because $0(0\cdot1)\neq 0\cdot 1$; $C15$ fails because $0(1(1\cdot1))\neq((0\cdot1)1)1$; $B12$ fails because $0(0(0\cdot1))\neq0((0\cdot0)1)$.
\end{ex}

\begin{ex}\label{ex:tab2b} Figure~\ref{fig:tab2}(b) is a 2-semilattice which is not in $\class{X}$. $A24$ fails because $0((0\cdot1)2)\neq(0(0\cdot1))2$.
\end{ex}

%

\begin{ex}\label{ex:tab3a} Figure~\ref{fig:tab3}(a) is member of $\class{X}$ which is not a member of $\class{T}_2$ or $\class{S}_2$, and is also not a semilattice. $C15$ fails because $0(1(1\cdot2))\neq((0\cdot1)1)2$. $B12$ fails because $0(1(0\cdot2))\neq0((1\cdot0)2)$. Associativity fails because $(0\cdot1)2\neq0(1\cdot2)$.
\end{ex}

\begin{figure}[t] 
  \centering
  \begin{minipage}[b]{90pt}
    \begin{center}
      \begin{tabular}{c|ccc}
      & 0 & 1 & 2\\ \hline
      0 & 0 & 2 & 1\\
      1 &  & 1 & 1\\
      2 &  &  & 2 
      \end{tabular}\\[3pt]
      (a) Example~\ref{ex:tab2a}
    \end{center}
    \end{minipage}
    \qquad\qquad
    \begin{minipage}[b]{90pt}
      \begin{center}

      \begin{tabular}{c|ccc}
      & 0 & 1 & 2\\ \hline
      0 & 0 & 1 & 0 \\
      1 &  & 1 & 2 \\
      2 &  &  & 2 
      \end{tabular}\\[3pt]
      (b) Example~\ref{ex:tab2b}
    \end{center}
    \end{minipage}
        \caption{Tables for Examples~\ref{ex:tab2a} and~\ref{ex:tab2b}}
  \label{fig:tab2}
\end{figure}

\begin{figure}[t] 
  \centering
  \begin{minipage}[b]{120pt}
    \begin{center}
    \begin{tabular}{c|cccc}
    & 0 & 1 & 2 & 3\\ \hline
    0 & 0 & 3 & 2 & 3 \\
    1 &  & 1 & 2 & 3 \\
    2 &  &  & 2 & 3 \\
    3 & & & & 3
    \end{tabular}\\[3pt]
      (a) Example~\ref{ex:tab3a}
    \end{center}
    \end{minipage}
    \qquad\qquad
    \begin{minipage}[b]{120pt}
      \begin{center}

      \begin{tabular}{c|cccccc}
      & 0 & 1 & 2 & 3 & 4 & 5\\ \hline
      0 & 0 & 0 & 0 & 4 & 5 & 4 \\
      1 &  & 1 & 3 & 2 & 5 & 4 \\
      2 &  &  & 2 & 1 & 5 & 4 \\
      3 &  &  &  & 3 & 0 & 5 \\
      4 &  &  &  &  & 4 & 0 \\
      5 &  &  &  &  &  & 5
      \end{tabular}\\[3pt]
      (b) Example~\ref{ex:t2nott1}
    \end{center}
    \end{minipage}
        \caption{Tables for Examples~\ref{ex:tab3a} and~\ref{ex:t2nott1}}
  \label{fig:tab3}
\end{figure}

\begin{ex}\label{ex:t2nott1} Figure~\ref{fig:tab3}(b) is a member of $\class{T}_2$ which is not in $\class{T}_1$. $A14$ fails because $0(0(1\cdot2))\neq(0(0\cdot1))2$.
\end{ex}

\begin{ex} \label{ex:3eltsquag} Figure~\ref{fig:tab4}(a) is a member of $\class{T}_1$ which is neither a 2-semilattice, nor a member of $\class{S}_2$, and hence is not a semilattice. The 2-semilattice law fails because $0(0\cdot1)\neq(0\cdot0)1$, while $B12$ fails because $0(0(0\cdot1))\neq0((0\cdot0)1)$.
\end{ex}

\begin{ex}\label{ex:tab4b} Figure~\ref{fig:tab4}(b) is a member of $\class{S}_2$ which is not a member of $\class{S}_1$. $B13$ fails because $0(1(0\cdot1))\neq(0\cdot1)(0\cdot1)$.
\end{ex}

\begin{ex}\label{ex:s13elt}  Figure~\ref{fig:tab4}(c) is a member of $\class{S}_1$ which is neither a 2-semilattice, nor a member of $\class{T}_2$, and hence is not a semilattice. The 2-semilattice law fails because $0(0\cdot1)\neq0\cdot1$, while $C15$ fails because $0(0(0\cdot1))\neq((0\cdot0)0)1$.
\end{ex}

While the Hasse diagram presented in Figure~\ref{fig:bmhasse} is not likely to be a lattice, we note that all of the intersections are true---that is, $\class{2SL}\cap\class{T}_2=\class{2SL}\cap\class{S}_2=\class{T}_2\cap\class{S}_2=\class{SL}$.


\begin{figure}[t]
  \centering
  \begin{minipage}{90pt}
    \begin{center}
    \begin{tabular}{c|ccc}
    & 0 & 1 & 2\\ \hline
    0 & 0 & 2 & 1\\
    1 &  & 1 & 0\\
    2 &  &  & 2 
    \end{tabular}\\[3pt]
      (a) Example~\ref{ex:3eltsquag}
    \end{center}
    \end{minipage}
    \qquad\qquad
    \begin{minipage}{90pt}
      \begin{center}
      \begin{tabular}{c|cccc}
      & 0 & 1 & 2 & 3\\ \hline
      0 & 0 & 2 & 3 & 3 \\
      1 &  & 1 & 3 & 3 \\
      2 &  &  & 2 & 3 \\
      3 & & & & 3
      \end{tabular}\\[3pt]
      (b) Example~\ref{ex:tab4b}
    \end{center}
    \end{minipage}
    \qquad\qquad
    \begin{minipage}{90pt}
      \begin{center}
      \begin{tabular}{r|ccc}
      & 0 & 1 & 2\\ \hline
      0 & 0 & 2 & 0\\
      1 &  & 1 & 1\\
      2 &  &  & 2 
      \end{tabular}\\[3pt]
      (c) Example~\ref{ex:s13elt}
    \end{center}
    \end{minipage}
        \caption{Tables for Examples~\ref{ex:3eltsquag},~\ref{ex:tab4b} and~\ref{ex:s13elt}}

  \label{fig:tab4}
\end{figure}

\subsection{Properties of Bol-Moufang CI-groupoids}

Our analysis thus far has determined properties of several, but not all of the varieties of CI-groupoids of Bol-Moufang type. In Theorem~\ref{thm:bm2sml} we showed that each of the listed identities was equivalent to the 2-semilattice law. Since $\class{X}$ is a subvariety of $\class{2SL}$, it is also a variety of 2-semilattices. Likewise, we showed in Theorem~\ref{thm:bmsl} that all of the listed identities are equivalent to the associative law, and thus determine the variety of semilattices. Following from the result of Bulatov \cite{Bulatov2006}, we know all three of these varieties ($\class{SL}$, $\class{2SL}$, and $\class{X}$) to be tractable. That the variety $\class{C}$ is indeed the variety of all CI-groupoids follows from the fact that $B45$, $D24$, $E12$ are immediate consequences of commutativity. The remainder of this section, as well as the next, is devoted to the other four varieties.

Using the Universal Algebra Calculator \cite{UACalc}, in conjunction with Mace4 \cite{prover9-mace4}, we investigated Maltsev conditions satisfied by the varieties $\class{T}_1$ and $\class{T}_2$. Using Mace4, we generated the only three element algebra in $\class{T}_2\setminus\class{T}_1$ (Example~\ref{ex:3eltsquag}), and provided it as input to the Universal Algebra Calculator. For this algebra, the Calculator did not find a majority, Pixley, or near-unanimity term, or terms for congruence distributivity, congruence join semi-distributivity, or congruence meet semi-distributivity. We then generated a 4-element algebra satisfying $A14$, for which the UA Calculator found only the Taylor term $x\cdot y$, inspiring our names for $\class{T}_1$ and $\class{T}_2$. Since $\class{Sl}\subseteq\class{T}_1\subseteq\class{T}_2$, these varieties are not congruence modular (so they fail to have few subpowers). The algebra $\A$ in Example~\ref{ex:3eltsquag} is a Latin square, and hence $\op{V}(\A)$ is congruence modular. However, $\Con(\A^2)\cong \mathbf{M}_4$, which is nondistributive. It follows that $\op{V}(\A)$ fails to be $\SD(\wedge)$, as do $\class{T}_1$ and $\class{T}_2$.

Finally, we performed a similar computer-aided analysis of $\class{S}_1$ and $\class{S}_2$. We generated the sole three element nonassociative groupoid occurring in these varieties using Mace4 (see Example~\ref{ex:s13elt}), and tested it for certain Maltsev conditions. For this algebra, the Universal Algebra Calculator produced WNU(4) and WNU(3) terms $w(x,y,z,u)=(xy)(zu)$ and $s(x,y,z)=(xy)(z(xy))$. These turned out to be $\SD(\wedge)$ terms for both varieties.

\begin{thm}\label{thm:S2SDM}
Every finite algebra in $\class{S}_2$ is congruence meet-semidistributive.
\end{thm}
\begin{proof}
Let $v(x,y,z)=(xy)(z(xy))$ and $w(x,y,z,u)=(xy)(zu)$. In any CI-groupoid, it is easily seen that $w(y,x,x,x)\=w(x,y,x,x)\=w(x,x,y,x)\=w(x,x,x,y)\=x(xy)$, so $w$ is a weak near-unanimity term. Using a similar argument, $v(y,x,x)\=v(x,y,x)\=(xy)(x(xy))$ and $v(x,x,y)\=x(xy)$. To see that $v$ is a weak near-unanimity term (and that $v(y,x,x)\=w(y,x,x,x)$), we just need to verify that $x(xy)\=(xy)(x(xy))$ holds in $\class{S}_2$. By \id{B12}{x(y(xz))}{x((yx)z)}, which is one of the defining identities for $\class{S}_2$, we have:

\begin{equation} \label{eqn:SDM1} x(yx)\=[x(yx)][y(x(yx))] \end{equation}
\begin{proof}[Proof of \eqref{eqn:SDM1}]   $\begin{aligned}[t]
x(yx) &\= [x(yx)][x(yx)] \= [x(yx)][x((yx)(yx))] \\
            &\= [x(yx)][x(y(x(yx)))] \= [x(yx)][(y(x(yx)))x]\\
            &\= [x(yx)][y((x(yx))x)] \= [x(yx)][y(x(x(yx)))]\\
            &\= [x(yx)][y(x((yx)x))] \= [x(yx)][y(x(y(xx)))]\\
            &\= [x(yx)][y(x(yx))]
  \end{aligned}$
  \renewcommand{\qedsymbol}{}
  \end{proof}

\begin{equation} \label{eqn:SDM2} [x(yx)][(zx)(x(yx))] \= [x(yx)][z(x(yx))] \end{equation}
\begin{proof}[Proof of \eqref{eqn:SDM2}] $\begin{aligned}[t]
[x(yx)][(zx)(x(yx))] &\= [x(yx)][(zx)(x(y(xx)))] \\&\= [x(yx)][(zx)(x((yx)x))]\\
                                  &\= [x(yx)][(zx)((x(yx))x)] \\&\= [x(yx)][(x(x(yx)))(zx)]\\
                                  &\= [x(yx)][x((x(yx))(zx))] \\&\= [x(yx)][x((zx)(x(yx)))]\\
                                  &\= [x(yx)][x(z(x(x(yx))))] \\&\= [x(yx)]x(z(x((yx)x)))]\\
                                  &\= [x(yx)][x(z(x(y(xx))))] \\&\= [x(yx)][x(z(x(yx)))]\\
                                  &\= [x(yx)][x((x(yx))z)] \\&\= [x(yx)][(x(x(yx)))z]\\
                                  &\= [x(yx)][z(x((yx)x))]\\&\=[x(yx)][z(x(y(xx)))]\\
                                  &\= [x(yx)][z(x(yx))]
  \end{aligned}$
  \renewcommand{\qedsymbol}{}
  \end{proof}

\begin{equation} \label{eqn:SDM3}x(xy)\=(xy)(x(xy))\end{equation}
\begin{proof}[Proof of \eqref{eqn:SDM3}] $\begin{aligned}[t] 
x(xy) &\= x(yx) \overset{\eqref{eqn:SDM1}}{\=} [x(yx)][y(x(yx))]\\
            &\overset{\eqref{eqn:SDM2}}\= [x(yx)][(yx)(x(yx))] \= [(yx)(x(yx))][x(yx)]\\
            &\= [(yx)(x(yx))][x((yx)(yx))]\\
            &\= [(yx)(x(yx))][x(y(x(yx)))]\\
            &\= [(yx)(x(yx))][x(y(x(y(xx))))]\\
            &\= [(yx)(x(yx))][x(y(x((yx)x)))]\\
            &\= [(yx)(x(yx))][x(y(x(x(yx))))]\\
            &\= [(yx)(x(yx))][x((yx)(x(yx)))]\\
            &\overset{\eqref{eqn:SDM1}}\=(yx)(x(yx))\=(xy)(x(xy))
  \end{aligned} $
  \renewcommand{\qedsymbol}{}
  \end{proof}  
Having justified \eqref{eqn:SDM3}, we conclude that $v$ is a WNU term, and the result follows from Theorem~\ref{thm:SDMterms}.
\end{proof}

\begin{ex}
While $B12$, together with commutativity and idempotence, is sufficient to prove \id{ \eqref{eqn:SDM3}}{x(xy)}{(xy)(x(xy))}, the equation does not hold for all CI-groupoids. For example, in the 3-element groupoid in Example~\ref{ex:3eltsquag}, $0(0\cdot1)\neq(0\cdot1)(0(0\cdot1))$.
\end{ex}

Following immediately from Theorem~\ref{thm:S2SDM} and Theorem~\ref{thm:SDMtractable}, we have the following corollary.

\begin{cor}
$\class{S}_2$ is tractable.
\end{cor}

\section{The structure of $\class{T}_1$ and $\class{T}_2$}\label{sec:Structure}

\subsection{Preliminaries}

Recall that $\class{T}_1$ is the variety of commutative, idempotent groupoids axiomatized by the additional identity $\id{A14}{x(x(yz))}{(x(xy))z}$. $\class{T}_1$ is contained in the variety $\class{T}_2$ defined by \id{C15}{x(y(yz))}{((xy)y)z}. Recall also that $xy$ is a Taylor term for both $\class{T}_1$ and $\class{T}_2$, but neither variety satisfies any familiar Maltsev conditions. As such, the few subpowers and bounded width algorithms cannot be used to solve the CSP over an arbitrary algebra from $\class{T}_1$ or $\class{T}_2$. As it turns out, we may use our main result to obtain the tractability of both, and additionally we obtain a strong structure theory for $\class{T}_1$. To prove that $\class{T}_2$ is tractable, we need a few lemmas, following which we give a pseudopartition operation for the variety.

\begin{lem}\label{lem:T2idents1} The variety $\class{T}_2$ satisfies the following identities:
\begin{align}
    \label{eq:t2-lem20} x(y(yx))&\=y(yx) \\ 
    \label{eq:t2-lem43} x(y(x(x(y(x(xz))))))&\=x(y(yz)) \\ 
    \label{eq:t2-lem79} x(y(yz))&\=x(y(y(x(xz)))) \\ 
    \label{eq:t2-lem82} (xy)(x(xz))&\=(xy)z \\ 
    \label{eq:t2-lem145}  x[y(y(z(zu)))]&\=x[(yz)(u(yz))] \\ 
    \label{eq:t2-lem172}  x(y(z(z(y(z(zu))))))&\=x(y(y(z(zu)))) \\ 
    \label{eq:t2-lem174} x(y(x(z(zy))))&\=z(z(y(yx)))  \\ 
    \label{eq:t2-lem236} x(y(y(z(y(yx)))))&\=x(z(y(yx))) \\ 
    \label{eq:t2-lem341} (x(y(yz)))(y(yu))&\=(x(y(yz)))u \\ 
    \label{eq:t2-lem384} x(y(y(z(zx))))&\=y(y(z(zx))) \\ 
    \label{eq:t2-lem387} (xy)(z(xy))&\=y(y(x(xz))) \\ 
    \label{eq:t2-lem389} x(x(y(yz)))&\=y(y(x(xz)))  
\end{align}
\end{lem}
\begin{proof}
See Appendix.
\end{proof}

\begin{lem}\label{lem:T2idents2}
The variety $\class{T}_2$ satisfies the identity 
\begin{equation}\label{eq:t2-lemL}
x(x(y(yz)))\=(y(xy))(z(y(xy))).
\end{equation}
\end{lem}
\begin{proof}
See Appendix.
\end{proof}

\begin{thm}\label{thm:t2-pseudo}
$x\vee y=y(xy)$ is a pseudopartition operation for $\class{T}_2$.
\end{thm}
\begin{proof}
See Appendix.
\end{proof}

\begin{defn}
A CI-groupoid satisfying $x(xy)\=y$ is called a \emph{squag} or \emph{Steiner quasigroup}.
\end{defn}

The quasigroup label is justified as the equation $ax=b$ has the unique solution $x=ab$ in any squag. Squags completely capture Steiner triple systems from combinatorics in an algebraic framework. A brief survey is presented in \cite[Chapter 3]{BurrisSanka}, while a more detailed exploration of squags and related objects can be found \cite{QuackenbushSquag}. As a variety of quasigroups, the variety of squags is congruence permutable. In fact, $q(x,y,z)= y(xz)$ is a Maltsev term. As we explained just after Theorem 2.14, this implies that the variety of squags is tractable. However, this argument cannot be extended to the variety $\class{T}_1$ (or $\class{T}_2$). Let $\A$ denote the groupoid displayed in Figure~\ref{fig:tab4}(a). $\A$ is the unique 3-element squag. The algebra $\A^\infty$ (see Definition~\ref{def:Ainfty}) is easily seen to lie in $\class{T}_1$. However $\A^\infty$ has a 2-element semilattice as a homomorphic image (identifying all 3 elements of $\A$). Consequently, $\class{T}_1$ cannot possess an edge term, so Theorem~\ref{thm:BWtractable} does not apply. On the other hand, the congruence lattice of $\A^2$ is not meet-semidistributive, so we cannot appeal to Theorem~\ref{thm:SDMtractable} to establish the tractability of $\class{T}_1$ (or by extension to $\class{T}_2$). 

Thus, neither of the two known tractability conditions can be applied to $\class{T}_2$. Nevertheless, $\class{T}_2$ is tractable. To establish this, we shall use Theorem~\ref{thm:pseudo-tractable}. 

\begin{cor}\label{cor:t2-tractable}
$\class{T}_2$ is tractable.
\end{cor}
\begin{proof}
Let $\A$ be a finite member of $\class{T}_2$. We showed in Theorem~\ref{thm:t2-pseudo} that $x\vee y = y(xy)$ is a pseudopartition operation for $\class{T}_2$. From the discussion following Theorem~\ref{thm:plonkaequivalence}, each P{\l}onka fibers satisfies $x\=x\vee y\=y(xy)$. Thus each block of the semilattice replica congruence lies in the variety of squags. Therefore, by Theorem~\ref{thm:pseudo-tractable}, $\A$ is tractable.
\end{proof}

This completes our our proof of the tractability of all varieties of CI-groupoids of Bol-Moufang type, with the exception of the variety $\class{C}$ of all CI-groupoids. We can obtain a still stronger result regarding the structure of $\class{T}_1$. Let $\Sigma=\left\{ xx\=x, xy\=yx, x(x(yz))\=(x(xy))z \right\}$, and let $x\vee y = y(xy)$ be the pseudopartition operation for $\class{T}_2$. Note that $\class{T}_1=\op{Mod}\left(\Sigma\right)$.  Define $\class{W}=\op{Mod}\left(\Sigma\cup\{x\vee y\=x\}\right)$.

As noted above, the variety of squags is the variety of CI-groupoids satisfying $x(xy)\=x(yx)\=y$. From the squag identity, we can easily derive $A14$: $x(x(yz))\=yz\=(x(xy))z$, which immediately gives:

\begin{lem}
$\class{W}$ is the variety of squags.
\end{lem}

We will show that $\class{T}_1$ is actually the regularization of $\class{W}$, following from Theorem~\ref{thm:plonkaequivalence}, by proving that $x\vee y$ is a partition operation for $\class{T}_1$. 

\begin{thm}
The variety $\class{T}_1$ is the regularization of the variety of squags.
\end{thm}
\begin{proof}
Let $\class{W}$ be the variety of squags as defined above. To prove that 
$\class{T}_1=\widetilde{\vphantom{t}\smash{\class{W}\,}}\!$, it suffices to show that $\Sigma$ can be used to derive each of the identities in Theorem~\ref{thm:plonkaequivalence}\eqref{eq:plonkaidentities}. Since \eqref{eq:plonka1}--\eqref{eq:plonka4} are shown in Theorem~$\ref{thm:t2-pseudo}$, and $\class{T}_1$ is a subvariety of $\class{T}_2$, we need only justify identity $\eqref{eq:plonka5}$: $(xy)\vee z \= (x\vee z)(y\vee z)$. As before, we do not label idempotence or commutativity.
\begin{align*}
(xy)\vee z &\overset{}{\=} z((xy)z)
                  \overset{}{\=} z(z(yx))
                  \overset{}{\=}z((z(zz))(yx))\\
                  &\overset{A14}{\=}z(z(z(z(yx))))
                  \overset{A14}{\=}z(z((z(zy))x))
                  \overset{}{\=}z(z(x(z(zy))))\\
                  &\overset{A14}{\=}(z(zx))(z(zy))
                  \overset{}{\=}(z(xz))(z(yz))
                  \overset{}{\=}(x\vee z)(y\vee z)\qedhere
\end{align*}
\end{proof}

As a consequence of this theorem, every member of $\class{T}_1$ is a P{\l}onka sum of squags. The term $x\vee y=y(xy)$ is, however, not a partition operation for $\class{T}_2$. Example~\ref{ex:t2nott1} is an algebra in $\class{T}_2$ for which the given pseudopartition operation fails to satisfy $\eqref{eq:plonka5}$, and so the algebras in $\class{T}_2$ need not be P{\l}onka sums, although they will decompose as disjoint unions of squags.

\section{Other varieties of CI-groupoids}\label{sec:OtherVarieties}

In the previous sections we have analyzed, as far as possible with current techniques, the tractability of the varieties of CI-groupoids of Bol-Moufang type. We continue the CSP-focused analysis of CI-groupoids by studying other weakenings of associativity.

One such identity, often studied in the presence of commutativity and idempotence, is the distributive law $x(yz) \= (xy)(xz)$. We will refer to the variety of commutative, idempotent distributive groupoids as the variety of CID-groupoids. They are, in some sense, the ``end of the line'' for our inquiry. In their booklet \cite{JezekKepkaNemec}, summarizing the state of the art in distributive groupoids,  Je{\v{z}}ek, Kepka, and N{\v{e}}mec share their opinion that ``the deepest non-associative theory within the framework of groupoids'' is the theory of distributive groupoids.

Another identity we will consider is the entropic law $(xy)(zw)\=(xz)(yw)$. In the literature this is sometimes referred to as mediality or the abelian law. A complete description of the lattice of subvarieties of commutative, idempotent, entropic groupoids (which we will call \emph{CIE-groupoids}) is given in \cite[Theorem~4.9]{JezekKepkaCIA}. Every idempotent, entropic groupoid (and hence every CIE-groupoid) is distributive. In \cite{KepkaNemec1981}, Kepka and N{\v{e}}mec show that every CID-groupoid which is not entropic has cardinality at least 81, so for the more general case of CID-groupoids, generating models and inspecting them for patterns is no longer a reasonable approach. Fortunately, P{\l}onka sums again prove useful. 

\begin{thm}[{\cite[Proposition~5.1]{Kepka1978}}]\label{thm:kepka51}
  Let $\A$ be a subdirectly irreducible CID-groupoid. Then there
  is a cancellation groupoid $\B$ such that either $\A
  \cong \B$ or $\A \cong \B^{\infty}$. 
\end{thm}

In Theorem~\ref{thm:kepka51}, \B is a subalgebra of \A, so it is also a CID-groupoid. Also, if $A$ is finite, then so is $B$. In the finite case \B,
being cancellative, is a Latin square. 

Let $xy^2=(xy)y$ and recursively define $xy^{j+1}=(xy^j)y$. Let $n$ be a positive integer, and define \Vn\ to be the variety of all
CID-groupoids satisfying the identity $xy^n\approx x$. Note that by
taking $x/y= xy^{n-1}$ in $\Vn$ we have $(x/y)\cdot y \approx xy^n \approx
x$. Combining this observation with commutativity we conclude that $\Vn$
is term-equivalent to a variety of quasigroups. From our discussion in Sections~\ref{sec:Preliminaries} and~\ref{sec:Plonka}, $\Vn$ is a strongly irregular, tractable variety.

\begin{thm}\label{thm:CIDplonka}
Every finite CID-groupoid is a P{\l}onka sum of Latin squares.
\end{thm}
\begin{proof}
Suppose that \A\ is an arbitrary finite CID-groupoid. Let $m=|A|$ and set
$n=m!$. Write \A\ as a subdirect product of subdirectly irreducible
algebras, $\A_i$, for $i\in I$. By Theorem~\ref{thm:kepka51}, each
$\A_i$ is isomorphic to either $\B_i$ or to $\B_i^{\infty}$, for some
Latin square $\B_i$. Since $|B_i| \leq m$, it follows that $\B_i \in
\Vn$. Consequently both $\B_i$ and $\B_i^{\infty}$ lie in
$\widetilde{\vphantom{t}\smash{\Vn}}\!$. Thus $\A\in \widetilde{\vphantom{t}\smash{\Vn}}\!$, so by Theorem~\ref{thm:plonkaequivalence},
\A\ is a P\l onka sum of Latin squares. 
\end{proof}

\begin{cor}\label{cor:regpreservestractable}
Let $\class{V}$ be an idempotent, tractable variety. Then $\widetilde{\class{V}\,}$ is a tractable variety.
\end{cor}
\begin{proof}
Suppose that $\class{V}$ is idempotent and tractable. If $\class{V}$ is regular, then $\class{V}=\widetilde{\class{V}\,}$ so there is nothing to prove. It is easy to see that an idempotent, irregular variety is strongly irregular. The claim now follows from Theorems~\ref{thm:plonkaequivalence} and~\ref{thm:pseudo-tractable}.
\end{proof}

\begin{cor}\label{cor:CIDtractable}
The variety of CID-groupoids is tractable.
\end{cor}
\begin{proof}
By Theorem~\ref{thm:CIDplonka}, every finite CID-groupoid lies in $\widetilde{\vphantom{t}\smash{\Vn}}\!$ for some $n\in\omega$. By Corollary~\ref{cor:regpreservestractable}, 
 $\widetilde{\vphantom{t}\smash{\Vn}}\!$ is tractable.
\end{proof}

\begin{cor}
The variety of CIE-groupoids is tractable.
\end{cor}
\begin{proof}
Every idempotent, entropic groupoid is distributive, following from: $$x(yz)\=(xx)(yz)\=(xy)(xz).$$ The result is then immediate following Corollary~\ref{cor:CIDtractable}.
\end{proof}

Let $n$ be an odd integer and $k$ an integer such that $2k \equiv 1 \pmod{n}$. Define $x\cdot y = kx+ky \pmod{n}$. One easily verifies that this defines a CIE-groupoid, $\A_n$ on the set $\{0,1,\dots,n-1\}$. Since this variety is regular, it contains the groupoid $\A_n^\infty$ as well. But arguing as we did above Corollary~\ref{cor:t2-tractable}, $\A_n^2$ is not congruence meet-semidistributive, and $\A_n^\infty$ has a semilattice quotient, so neither Theorem~\ref{thm:SDMtractable} nor \ref{thm:BWtractable} can be used to demonstrate the tractability of the variety of CIE- (or CID-) groupoids. 

\section*{Appendix}

The appendix is available online at \url{http://www.arxiv.org}, or \url{http://orion.math.iastate.edu/cbergman/manuscripts/cigcsp.pdf}.


%


\appendix
\section*{Appendix}

We present the equational derivations justifying some of the results from Sections~\ref{sec:BMGroupoids} and \ref{sec:Structure}.

\begin{proof}[Proof of $(D23, I)\Rightarrow 2\class{SL}$] (From Lemma~\ref{lem:I2})

First, a lemma (later referenced as $L1$).: $((xy)x)x \= x((yx)x)$.
\begin{align*}
((xy)x)x &\overset{I, D23}{\=} (x((yx)x))x\\
  	     &\overset{I, D23}{\=} x((((yx)x)x)x)\\
  	     &\overset{I, D23}{\=} x((x(((yx)x)x))x)\\
  	     &\overset{D23}{\=} x(((x(yx))(xx))x)\\
  	     &\overset{I}{\=} x(((x(yx))x)x)\\
  	     &\overset{I, D23}{\=} x((x((x(yx))x))x)\\
  	     &\overset{D23, I}{\=}x((x((yx)x))x)\\
  	     &\overset{D23, I}{\=} x(((yx)x)x)\\
  	     &\overset{I, D23}{\=} x((x((yx)x))x)\\
  	     &\overset{D23}{\=} x(((xy)(xx))x)\\
  	     &\overset{I}{\=} x(((xy)x)x)\\
  	     &\overset{I, D23}{\=} x((xx((xy)x))x)\\
  	     &\overset{D23, I}{\=} x((x(yx))x)\\
  	     &\overset{D23, I}{\=}x((yx)x)
\end{align*}

Another short lemma (later referenced as $L2$): $x((yx)x)\=x(yx)$.
\begin{align*}
x((yx)x) &\overset{I, D23}{\=} x((x(yx))x)\\
             &\overset{I, D23}{\=} x((x((xy)x))x)\\
             &\overset{D23, I}{\=} x(((xy)x)x)\\
             &\overset{L1}{\=} x(x((yx)x))\\
             &\overset{D23, I}{\=} x((xy)x)\\
             &\overset{D23, I}{\=} x(yx)
\end{align*}

The 2-semilattice law: $xy\=x(xy)$.
\begin{align*}
xy &\overset{I, D23}{\=} ((xy)x)(y(xy))\\
     &\overset{I, D23}{\=} (x((yx)x))(y(xy))\\
     &\overset{L2}{\=} (x(yx))(y(xy))\\
     &\overset{I, D23}{\=} ((xy)(yx))(y(xy))\\
     &\overset{I, D23}{\=} ((xy)(yx))((yx)(xy))\\
     &\overset{D23, I}{\=} (xy)((yx)(xy))\\
     &\overset{L2}{\=} (xy)(((yx)(xy))(xy))\\
     &\overset{D23, I}{\=} ((xy)(yx))(xy)\\
     &\overset{D23, I}{\=} (x(yx))(xy)\\
     &\overset{L2}{\=} (x((yx)x))(xy)\\
     &\overset{D23, I}{\=} ((xy)x)(xy)\\
     &\overset{I}{\=} ((xy)x)((xy)(xy))\\
     &\overset{D23}{\=} (xy)((x(xy))(xy))\\
     &\overset{L2}{\=} (xy)(x(xy))\\
     &\overset{I, D23}{\=} ((xy)x)(x(xy))\\
     &\overset{I, D23}{\=} (x((yx)x))(x(xy))\\
     &\overset{L2}{\=} (x(yx))(x(xy))\\
     &\overset{I, D23}{\=} (x((xy)x))(x(xy))\\
     &\overset{I, D23}{\=} (x(x((yx)x)))(x(xy))\\
     &\overset{L2}{\=} x(x(yx)))(x(xy))\\
     &\overset{L2}{\=} (x(x((yx)x)))(x(xy))\\
     &\overset{D23}{\=} (x((xy)(xx)))(x(xy))\\
     &\overset{I}{\=} (x((xy)x))(x(xy))\\
     &\overset{L2}{\=} (x(((xy)x)x))(x(xy))\\
     &\overset{D23, I}{\=} ((x(xy))x)(x(xy))\\
     &\overset{I, D23}{\=} (x(xy))((x(x(xy)))(x(xy)))\\
     &\overset{L2}{\=} (x(xy))(x(x(xy)))\\
     &\overset{L2}{\=} (x(xy))((x(x(xy)))(x(xy)))\\
     &\overset{D23, I}{\=} ((x(xy))x)(x(xy))\\
     &\overset{I, D23}{\=} (x(((xy)x)x))(x(xy))\\
     &\overset{L2}{\=} (x((xy)x))(x(xy))\\
     &\overset{I, D23}{\=} ((x(xy))((xy)x))(x(xy))\\
     &\overset{I, D23}{\=} (x(xy))((((xy)x)(x(xy)))(x(xy)))\\
     &\overset{L2}{\=} (x(xy))(((xy)x)(x(xy)))\\
     &\overset{D23, I}{\=} (x(xy))((xy)(x(xy)))\\
     &\overset{D23, I}{\=} (x(xy))(((xy)x)(x(xy)))\\
     &\overset{D23, I}{\=} ((x(xy))((xy)x))(((xy)x)(x(xy)))\\
     &\overset{D23, I}{\=} ((x(xy))((xy)x))((xy)(x(xy)))\\
     &\overset{D23, I}{\=} (x((xy)x))((xy)(x(xy)))\\
     &\overset{L2}{\=} (x(((xy)x)x))((xy)(x(xy)))\\
     &\overset{D23, I}{\=} ((x(xy))x)((xy)(x(xy)))\\
     &\overset{D23, I}{\=} x(xy)\qedhere
\end{align*}
\end{proof}

\begin{proof}[Proof of Lemma~\ref{lem:T2idents1}]
\begin{align*}
x(y(yx))&\overset{}{\=}(y(yx))x\\
             &\overset{}{\=}(((yy)y)(yx))x\\
             &\overset{C15}{\=}(y(y(y(yx))))x\\
             &\overset{}{\=}(((y(yx))y)y)x\\
             &\overset{C15}{\=}(y(yx))(y(yx))\\
             &\overset{}{\=}y(yx)
\end{align*}

\begin{align*}
x(y(x(x(y(x(xz))))))&\overset{C15}{\=}x(((yx)x)(y(x(xz))))\\
                                &\overset{C15}{\=}x(((yx)x)(((yx)x)z))\\
                                &\overset{C15}{\=}((x((yx)x))((yx)x))z\\
                                &\overset{}{\=}((x(x(xy)))((yx)x))z\\
                                &\overset{C15}{\=}((((xx)x)y)((yx)x))z\\
                                &\overset{}{\=}((xy)((yx)x))z\\
                                &\overset{}{\=}(((yx)x)(xy))z\\
                                &\overset{C15}{\=}(y(x(x(xy))))z\\
                                &\overset{C15}{\=}(y(((xx)x)y))z\\
                                &\overset{}{\=}(y(xy))z\\
                                &\overset{}{\=}((xy)y)z\\
                                &\overset{C15}{\=}x(y(yz))
\end{align*}

\begin{align*}
x(y(yz))&\overset{C15}{\=}((xy)y)z\\
             &\overset{}{\=}(y(xy))z\\
             &\overset{}{\=}[[(y(xy))(y(xy))](y(xy))]z\\
             &\overset{C15}{\=}(y(xy))[(y(xy))[(y(xy))z]]\\
             &\overset{}{\=}((xy)y)[((xy)y)[((xy)y)z]]\\
             &\overset{C15}{\=}x[y[y[((xy)y)[((xy)y)z]]]]\\
             &\overset{C15}{\=}x[y[[(y((xy)y))((xy)y)]z]]\\
             &\overset{}{\=}x[y[[(y(y(yx)))((xy)y)]z]]\\
             &\overset{C15}{\=}x[y[[(((yy)y)x)((xy)y)]z]]\\
             &\overset{}{\=}x[y[[(yx)((xy)y)]z]]\\
             &\overset{}{\=}x[y[[((xy)y)(yx)]z]]\\
             &\overset{C15}{\=}x[y[[x(y(y(yx)))]z]]\\
             &\overset{C15}{\=}x[y[[x(((yy)y)x)]z]]\\
             &\overset{}{\=}x[y[[x(yx)]z]]\\
             &\overset{}{\=}x[y[[(yx)x]z]]\\
             &\overset{C15}{\=}x(y(y(x(xz))))
\end{align*}

\begin{align*}
(xy)(x(xz))&\overset{C15}{\=}(((xy)x)x)z\\
                  &\overset{}{\=}(x(x(xy)))z\\
                  &\overset{C15}{\=}(((xx)x)y)z\\
                  &\overset{}{\=}(xy)z
\end{align*}

\begin{align*}
x[y(y(z(zu)))]&\overset{\eqref{eq:t2-lem43}}{\=}x[y(y(z(y(y(z(y(yu)))))))]\\
                      &\overset{C15}{\=}x[y(y(((zy)y)(z(y(yu)))))]\\
                      &\overset{C15}{\=}x[y(y(((zy)y)(((zy)y)u)))]\\
                      &\overset{C15}{\=}x[y(((y((zy)y))((zy)y))u)]\\
                      &\overset{}{\=}x[y(((y(yz))(y(y(yz))))u)]\\
                      &\overset{C15}{\=}x[y(((y(yz))(((yy)y)z))u)]\\
                      &\overset{}{\=}x[y(((y(yz))(yz))u)]\\
                      &\overset{C15}{\=}x[y(y((yz)((yz)u)))]\\
                      &\overset{C15}{\=}((xy)y)[(yz)((yz)u)]\\
                      &\overset{C15}{\=}[(((xy)y)(yz))(yz)]u\\
                      &\overset{C15}{\=}[(x(y(y(yz))))(yz)]u\\
                      &\overset{C15}{\=}[(x(((yy)y)z))(yz)]u\\
                      &\overset{}{\=}[(x(yz))(yz)]u\\
                      &\overset{C15}{\=}x[(yz)((yz)u)]\\
                      &\overset{}{\=}x[(yz)(u(yz))]
\end{align*}

\begin{align*}
x(y(z(z(y(z(zu))))))&\overset{C15}{\=}x(((yz)z)(y(z(zu))))\\
                                &\overset{C15}{\=}x(((yz)z)(((yz)z)u))\\
                                &\overset{C15}{\=}((x((yz)z))((yz)z))u\\
                                &\overset{}{\=}(((yz)z)(x(z(zy))))u\\
                                &\overset{\eqref{eq:t2-lem20}}{\=}(((yz)z)(x(y(z(zy)))))u\\
                                &\overset{\eqref{eq:t2-lem20}}{\=}(((yz)z)(x(y(y(z(zy))))))u\\
                                &\overset{}{\=}((x(y(y(z(zy)))))(z(zy)))u\\
                                &\overset{C15}{\=}((((xy)y)(z(zy)))(z(zy)))u\\
                                &\overset{C15}{\=}((xy)y)((z(zy))((z(zy))u))\\
                                &\overset{C15}{\=}x(y(y((z(zy))((z(zy))u))))\\
                                &\overset{C15}{\=}x(y(((y(z(zy)))(z(zy)))u))\\
                                &\overset{\eqref{eq:t2-lem20}}{\=}x(y(((z(zy))(z(zy)))u))\\
                                &\overset{}{\=}x(y((z(zy))u))\\
                                &\overset{}{\=}x(y(((yz)z)u))\\
                                &\overset{C15}{\=}x(y(y(z(zu))))
\end{align*}

\begin{align*}
x(y(x(z(zy))))&\overset{}{\=}(y(x(z(zy))))x\\
                       &\overset{\eqref{eq:t2-lem20}}{\=}(y(z(y(z(zy)))))x\\
                       &\overset{\eqref{eq:t2-lem20}}{\=}(y(x(y(y(z(zy))))))x\\
                       &\overset{C15}{\=}(y(((xy)y)(z(zy))))x\\
                       &\overset{}{\=}(y((y(yx))(z(zy))))x\\
                       &\overset{C15}{\=}(y((((y(yx))z)z)y))x\\
                       &\overset{}{\=}(((z((y(yx))z))y)y)x\\
                       &\overset{C15}{\=}(z((y(yx))z))(y(yx))\\
                       &\overset{}{\=}(((y(yx))z)z)(y(yx))\\
                       &\overset{C15}{\=}(y(yx))(z(z(y(yx))))\\
                       &\overset{\eqref{eq:t2-lem20}}{\=}z(z(y(yx)))
\end{align*}

\begin{align*}
x(y(y(z(y(yx)))))&\overset{C15}{\=}((xy)y)(z(y(yx)))\\
                           &\overset{}{\=}(y(yx))((y(yx))z)\\
                           &\overset{}{\=}((y(yx))(y(yx)))((y(yx))z)\\
                           &\overset{\eqref{eq:t2-lem20}}{\=}((x(y(yx)))(y(yx)))((y(yx))z)\\
                           &\overset{C15}{\=}x((y(yx))((y(yx))((y(yx))z)))\\
                           &\overset{C15}{\=}x((((y(yx))(y(yx)))(y(yx)))z)\\
                           &\overset{}{\=}x((y(yx))z)\\
                           &\overset{}{\=}x(z(y(yx)))
\end{align*}

\begin{align*}
(x(y(yz)))(y(yu))&\overset{\eqref{eq:t2-lem82}}{\=}(x(y(yz)))(x(x(y(yu))))\\
                           &\overset{\eqref{eq:t2-lem172}}{\=}(x(y(yz)))(x(y(y(x(y(yu))))))\\
                           &\overset{C15}{\=}(x(y(yz)))(((xy)y)(x(y(yu))))\\
                           &\overset{C15}{\=}(x(y(yz)))(((xy)y)(((xy)y)u))\\
                           &\overset{C15}{\=}(((xy)y)z)(((xy)y)(((xy)y)u))\\
                           &\overset{\eqref{eq:t2-lem82}}{\=}(((xy)y)z)u\\
                           &\overset{C15}{\=}(x(y(yz)))u
\end{align*}

\begin{align*}
x(y(y(z(zx))))&\overset{}{\=}x(y(((yy)y)(z(zx))))\\
                       &\overset{C15}{\=}x(y(y(y(y(z(zx))))))\\
                       &\overset{\eqref{eq:t2-lem145}}{\=}x(y(y((yz)(x(yz)))))\\
                       &\overset{C15}{\=}((xy)y)((yz)(x(yz)))\\
                       &\overset{}{\=}((xy)y)((yz)(x(((yy)y)z)))\\
                       &\overset{C15}{\=}((xy)y)((yz)(x(y(y(yz)))))\\
                       &\overset{C15}{\=}((xy)y)((yz)(((xy)y)(yz)))\\
                       &\overset{}{\=}(y(xy))((yz)((yz)(y(xy))))\\
                       &\overset{\eqref{eq:t2-lem20}}{\=}(yz)((yz)(y(xy)))\\
                       &\overset{}{\=}(yz)(((xy)y)(yz))\\
                       &\overset{C15}{\=}(yz)(x(y(y(yz))))\\
                       &\overset{C15}{\=}(yz)(x(((yy)y)z))\\
                       &\overset{}{\=}(yz)(x(yz))\\
                       &\overset{}{\=}(yz)(x(y(z(zz))))\\
                       &\overset{C15}{\=}(yz)(x(((yz)z)z))\\
                       &\overset{}{\=}(yz)((((yz)z)z)x)\\
                       &\overset{C15}{\=}(yz)((yz)(z(zx)))\\
                       &\overset{\eqref{eq:t2-lem20}}{\=}(z(zx))((yz)((yz)(z(zx))))\\
                       &\overset{C15}{\=}(z(zx))((yz)((((yz)z)z)x))\\
                       &\overset{C15}{\=}(z(zx))((yz)((y(z(zz)))x))\\
                       &\overset{}{\=}(z(zx))((yz)((yz)x))\\
                       &\overset{C15}{\=}(((z(zx))(yz))(yz))x\\
                       &\overset{}{\=}(((z(zx))(((yy)y)z))(yz))x\\
                       &\overset{C15}{\=}(((z(zx))(y(y(yz))))(yz))x\\
                       &\overset{C15}{\=}(((((z(zx))y)y)(yz))(yz))x\\
                       &\overset{C15}{\=}(((z(zx))y)y)((yz)((yz)x))\\
                       &\overset{C15}{\=}(z(zx))(y(y((yz)((yz)x))))\\
                       &\overset{C15}{\=}(z(zx))(y(((y(yz))(yz))x))\\
                       &\overset{}{\=}(z(zx))(y(((y(yz))(((yy)y)z))x))\\
                       &\overset{C15}{\=}(z(zx))(y(((y(yz))(y(y(yz))))x))\\
                       &\overset{}{\=}(z(zx))(y(((y(y(yz)))(y(yz)))x))\\
                       &\overset{C15}{\=}(z(zx))(y(y((y(yz))((y(yz))x))))\\
                       &\overset{}{\=}(z(zx))(y(y(((zy)y)(((zy)y)x))))\\
                       &\overset{C15}{\=}(z(zx))(y(y(((zy)y)(z(y(yx))))))\\
                       &\overset{C15}{\=}(z(zx))(y(y(z(y(y(z(y(yx))))))))\\
                       &\overset{\eqref{eq:t2-lem43}}{\=}(z(zx))(y(y(z(zx))))\\
                       &\overset{\eqref{eq:t2-lem20}}{\=}y(y(z(zx)))
\end{align*}

\begin{align*}
(xy)(z(xy))&\overset{}{\=}(xy)(z(x(y(yy))))\\
                  &\overset{C15}{\=}(xy)(z(((xy)y)y))\\
                  &\overset{}{\=}(xy)(z(y(y(xy))))\\
                  &\overset{C15}{\=}(xy)(((zy)y)(xy))\\
                  &\overset{}{\=}(xy)((xy)(y(zy)))\\
                  &\overset{\eqref{eq:t2-lem20}}{\=}(y(zy))((xy)((xy)(y(zy))))\\
                  &\overset{}{\=}(y(zy))((xy)((xy)(y(yz))))\\
                  &\overset{C15}{\=}(y(zy))((xy)((((xy)y)y)z))\\
                  &\overset{C15}{\=}(y(zy))((xy)((x(y(yy)))z))\\
                  &\overset{}{\=}(y(zy))((xy)((xy)z))\\
                  &\overset{C15}{\=}(((y(zy))(xy))(xy))z\\
                  &\overset{}{\=}((xy)((y(zy))(xy)))z\\
                  &\overset{}{\=}((xy)((y(zy))(((xx)x)y)))z\\
                  &\overset{C15}{\=}((xy)((y(zy))(x(x(xy)))))z\\
                  &\overset{C15}{\=}((xy)((((y(zy))x)x)(xy)))z\\
                  &\overset{}{\=}(((x((y(zy))x))(xy))(xy))z\\
                  &\overset{C15}{\=}(x((y(zy))x))((xy)((xy)z))\\
                  &\overset{}{\=}(((y(zy))x)x)((xy)((xy)z))\\
                  &\overset{C15}{\=}(y(zy))(x(x((xy)((xy)z))))\\
                  &\overset{C15}{\=}(y(zy))(x(((x(xy))(xy))z))\\
                  &\overset{}{\=}(y(zy))(x(((x(xy))(((xx)x)y))z))\\
                  &\overset{C15}{\=}(y(zy))(x(((x(xy))(x(x(xy))))z))\\
                  &\overset{}{\=}(y(zy))(x(((x(x(xy)))(x(xy)))z))\\
                  &\overset{C15}{\=}(y(zy))(x(x((x(xy))((x(xy))z))))\\
                  &\overset{}{\=}(y(zy))(x(x(((yx)x)(((yx)x)z))))\\
                  &\overset{C15}{\=}(y(zy))(x(x(((yx)x)(y(x(xz))))))\\
                  &\overset{C15}{\=}(y(zy))(x(x(y(x(x(y(x(xz))))))))\\
                  &\overset{\eqref{eq:t2-lem43}}{\=}(y(zy))(x(x(y(yz))))\\
                  &\overset{}{\=}((zy)y)(x(x(y(yz))))\\
                  &\overset{C15}{\=}z(y(y(x(x(y(yz))))))\\
                  &\overset{\eqref{eq:t2-lem79}}{\=}z(y(y(x(xz))))
\end{align*}

\begin{align*}
x(x(y(yz)))&\overset{\eqref{eq:t2-lem387}}{\=}(yx)(z(yx))\\
                  &\overset{}{\=}(xy)(z(xy))\\
                  &\overset{\eqref{eq:t2-lem387}}{\=}y(y(x(xz)))\qedhere
\end{align*}
\end{proof}

\begin{proof}[Proof of Lemma~\ref{lem:T2idents2}]
\begin{align*}
x(x(y(yz)))&\overset{\eqref{eq:t2-lem20}}{\=}(y(yz))(x(x(y(yz))))\\
                  &\overset{\eqref{eq:t2-lem384}}{\=}(y(yz))(z(x(x(y(yz)))))\\
                  &\overset{C15}{\=}(y(yz))(((zx)x)(y(yz)))\\
                  &\overset{}{\=}(y(yz))((y(yz))(x(xz)))\\
                  &\overset{\eqref{eq:t2-lem387}}{\=}(x(y(yz)))(z(x(y(yz))))\\
                  &\overset{\eqref{eq:t2-lem341}}{\=}(x(y(yz)))(y(y(z(x(y(yz))))))\\
                  &\overset{C15}{\=}(((x(y(yz)))y)y)(z(x(y(yz))))\\
                  &\overset{}{\=}(y(y(x(y(yz)))))(z(x(y(yz))))\\
                  &\overset{\eqref{eq:t2-lem82}}{\=}(y(y(x(y(yz)))))(y(y(z(x(y(yz))))))\\
                  &\overset{\eqref{eq:t2-lem236}}{\=}(y(y(x(y(yz)))))(y(y(z(y(y(x(y(yz))))))))\\
                  &\overset{C15}{\=}(y(y(x(y(yz)))))(y(y(((zy)y)(x(y(yz))))))\\
                  &\overset{}{\=}(y(y(x(y(yz)))))(y(y((x(y(yz)))(y(yz)))))\\
                  &\overset{\eqref{eq:t2-lem82}}{\=}(y(y(x(y(yz)))))(y(y((x(y(yz)))(x(x(y(yz)))))))\\
                  &\overset{}{\=}(y(y(x(y(yz)))))(y(y((x(y(yz)))((x(y(yz)))x))))\\
                  &\overset{\eqref{eq:t2-lem145}}{\=}(y(y(x(y(yz)))))((y(x(y(yz))))(x(y(x(y(yz))))))\\
                  &\overset{}{\=}(y(y(x(y(yz)))))((y(x(y(yz))))((y(x(y(yz))))x))\\
                  &\overset{C15}{\=}(((y(y(x(y(yz)))))(y(x(y(yz)))))(y(x(y(yz)))))x\\
                  &\overset{C15}{\=}(y((y(x(y(yz))))((y(x(y(yz))))(y(x(y(yz)))))))x\\
                  &\overset{}{\=}(y(y(x(y(yz)))))x\\
                  &\overset{}{\=}x(y(y(x(y(yz)))))\\
                  &\overset{C15}{\=}((xy)y)(x(y(yz)))\\
                  &\overset{C15}{\=}((xy)y)(((xy)y)z)\\
                  &\overset{}{\=}(y(xy))(z(y(xy)))\qedhere
\end{align*}
\end{proof}

\begin{proof}[Proof of Theorem~\ref{thm:t2-pseudo}]
We need to show that $x\vee y=y(xy)$ satisfies identities \eqref{eq:plonka1}--\eqref{eq:plonka4} in Theorem~\ref{thm:plonkaequivalence}. In order, they are:

\begin{align*}
\eqref{eq:plonka1}\,\,x\vee x &\= x:\\
x\vee x &\=x(xx)\\
              &\=x
\end{align*}

\begin{align*}
\eqref{eq:plonka2}\,\,x\vee(y\vee z) &\= (x\vee y)\vee z: \\
x\vee(y\vee z)&\overset{}{\=}x\vee(z(yz))\\
                         &\overset{}{\=}(z(yz))(x(z(yz)))\\
                         &\overset{}{\=}(z(zy))(x(z(zy)))\\
                         &\overset{}{\=}((z(zy))(x(z(zy))))((z(zy))(x(z(zy))))\\
                         &\overset{}{\=}((x(z(zy)))(z(zy)))((z(zy))(x(z(zy))))\\
                         &\overset{C15}{\=}x((z(zy))((z(zy))((z(zy))(x(z(zy))))))\\
                         &\overset{}{\=}x((z(zy))((z(zy))((z(zy))((z(zy))x))))\\
                         &\overset{C15}{\=}x((z(zy))((((z(zy))(z(zy)))(z(zy)))x))\\
                         &\overset{}{\=}x((z(zy))((z(zy))x))\\
                         &\overset{}{\=}x((x(z(zy)))(z(zy)))\\
                         &\overset{C15}{\=}x((((x(z(zy)))z)z)y)\\
                         &\overset{}{\=}x(y(z(z(x(z(zy))))))\\
                         &\overset{\eqref{eq:t2-lem236}}{\=}x(y(x(z(zy))))\\
                         &\overset{\eqref{eq:t2-lem174}}{\=}z(z(y(yx)))\\
                         &\overset{\eqref{eq:t2-lem79}}{\=}z(z(y(y(z(zx)))))\\
                         &\overset{\eqref{eq:t2-lem79}}{\=}z(z(y(y(z(z(y(yx)))))))\\
                         &\overset{C15}{\=}z(((zy)y)(z(z(y(yx)))))\\
                         &\overset{}{\=}z((y(zy))(z(z(y(yx)))))\\
                         &\overset{\eqref{eq:t2-lem145}}{\=}z((y(zy))((zy)(x(zy))))\\
                         &\overset{}{\=}z((y(zy))((zy)((zy)x)))\\
                         &\overset{C15}{\=}z((((y(zy))(zy))(zy))x)\\
                         &\overset{C15}{\=}z((y((zy)((zy)(zy))))x)\\
                         &\overset{}{\=}z((y(zy))x)\\
                         &\overset{}{\=}z(((zy)y)x)\\
                         &\overset{C15}{\=}z(z(y(yx)))\\
                         &\overset{}{\=}z((y(xy))z)\\
                         &\overset{}{\=}z((x\vee y)z)\\
                         &\overset{}{\=}(x\vee y)\vee z
\end{align*}

\begin{align*}
\eqref{eq:plonka3}\,\,x\vee(y\vee z) &\= x\vee(z\vee y): \\
x\vee(y\vee z)&\overset{}{\=}(y\vee z)(x(y\vee z))\\
                         &\overset{}{\=}(z(yz))(x(z(yz)))\\
                         &\overset{}{\=}((z(yz))(x(z(yz))))((z(yz))(x(z(yz))))\\
                         &\overset{}{\=}((x(z(yz)))(z(yz)))((z(yz))(x(z(yz))))\\
                         &\overset{C15}{\=}x((z(yz))((z(yz))((z(yz))(x(z(yz))))))\\
                         &\overset{C15}{\=}x((((z(yz))(z(yz)))(z(yz)))(x(z(yz))))\\
                         &\overset{}{\=}x((z(yz))(x(z(yz))))\\
                         &\overset{}{\=}x(((yz)z)(x(z(zy))))\\
                         &\overset{C15}{\=}x(y(z(z(x(z(zy))))))\\
                         &\overset{\eqref{eq:t2-lem236}}{\=}x(y(x(z(zy))))\\
                         &\overset{\eqref{eq:t2-lem174}}{\=}z(z(y(yx)))\\
                         &\overset{\eqref{eq:t2-lemL}}{\=}(y(zy))(x(y(zy)))\\
                         &\overset{}{\=}x\vee(y(zy))\\
                         &\overset{}{\=}x\vee(z\vee y)
\end{align*}

\begin{align*}
\eqref{eq:plonka4}\,\,x\vee(yz) &\= x\vee(y\vee z): \\
x\vee(yz)&\overset{}{\=}(yz)(x(yz))\\
	       &\overset{\eqref{eq:t2-lem387}}{\=}z(z(y(yx)))\\
	       &\overset{\eqref{eq:t2-lem389}}{\=}y(y(z(zx)))\\
	       &\overset{\eqref{eq:t2-lemL}}{\=}(z(yz))(x(z(yz)))\\
	       &\overset{}{\=}x\vee(z(yz))\\
	       &\overset{}{\=}x\vee(y\vee z)\qedhere
\end{align*}
\end{proof}

\end{document}